\documentclass[a4paper,UKenglish]{amsart}

\usepackage[table]{xcolor}
\definecolor{lightgray}{gray}{0.9}

\usepackage{thmtools}
\usepackage{thm-restate}
\usepackage{hyperref}
 \usepackage[capitalise,noabbrev]{cleveref}

\newtheorem{theorem}{\bfseries Theorem}[section]
\newtheorem{theoremx}{\bfseries Theorem}
\newtheorem{corollary}[theorem]{\bfseries Corollary}
\newtheorem{proposition}[theorem]{\bfseries Proposition}
\newtheorem{lemma}[theorem]{\bfseries Lemma}
\theoremstyle{definition}
\newtheorem{remark}[theorem]{\bfseries Remark}
\newtheorem{definition}[theorem]{\bfseries Definition}
\newtheorem{example}[theorem]{\bfseries Example}
\newtheorem{notation}[theorem]{Notation}

\crefname{theoremx}{theorem}{theorems}
\Crefname{theoremx}{Theorem}{Theorems}

\newcommand{\secref}[1]{\S\ref{#1}}

\usepackage{lineno}
\usepackage{url,amscd,epsfig,enumerate,graphicx,amsfonts,latexsym}
\usepackage{amssymb,amscd}
\usepackage{amsfonts,mathrsfs}
\renewcommand{\leq}{\leqslant}
\renewcommand{\geq}{\geqslant}
\usepackage{subcaption}

\usepackage{tikz}
\usetikzlibrary{decorations.markings}
\usetikzlibrary{arrows,snakes}
\usepackage{verbatim}
\tikzset{edge/.style = {->,> = latex}}
\usetikzlibrary{decorations.pathreplacing,shapes.multipart}
\usetikzlibrary{automata} 
\usetikzlibrary{positioning} 
\usetikzlibrary{arrows} 

\newcommand\OO{\Omega}
\newcommand{\Oh}{\mathcal{O}}

\renewcommand{\leq}{\leqslant}
\renewcommand{\geq}{\geqslant}

\newcommand{\abs}[1]{\left|\mathinner{#1}\right|}
\newcommand{\Abs}[1]{\left\Vert\mathinner{#1}\right\Vert}

\newenvironment{proof*}[1]
  {\renewcommand{\proofname}{#1}%
   \begin{proof}}
  {\end{proof}}

\newcommand{\PSPACE}{\ensuremath{\mathsf{PSPACE}}}
\newcommand{\NSPACE}{\ensuremath{\mathsf{NSPACE}}}

\usepackage{xcolor}

\newcommand{\DG}{Dahmani and Guirardel}
\newcommand{\Qier}{quasi-isometrically embeddable rational}
\newcommand{\qier}{qier}
\newcommand{\N}{\mathbb N}
\newcommand{\qgeod}{quasigeodesic}
\newcommand{\qgeods}{quasigeodesics}

\newcommand{\langT}{\mathcal T}

\newcommand{\control}{A}

\title[The complexity of solution sets to equations in hyperbolic groups]{The complexity of solution sets to \\equations in hyperbolic groups}

\author{Laura Ciobanu}
\address{School of Mathematical and Computer Sciences,
 Heriot-Watt University, 
 Edinburgh EH14 4AS,
 Scotland}
\email{l.ciobanu@hw.ac.uk}

\author{Murray Elder}
\address{School of Mathematical and Physical Sciences, University of Technology Sydney, Ultimo NSW 2007, Australia}
\email{murray.elder@uts.edu.au}

\date{\today}



\thanks{Research supported by Australian Research Council (ARC) Project DP160100486,   EPSRC grant EP/R035814/1, a Follow-On Grant from the International Centre of Mathematical Sciences (ICMS), Edinburgh, and an LMS Scheme 2 grant.}

\begin{document}
\maketitle

\begin{abstract}
We show that the full set of solutions to systems of equations and inequations in a hyperbolic group, as shortlex geodesic words
(or any regular set of quasigeodesic normal forms), 
  is  an EDT0L language whose 
  specification can be  computed in \NSPACE$(n^2\log n)$ for the torsion-free case and \NSPACE$(n^4\log n)$ in the torsion case. 
  Furthermore, in the presence of effective \Qier\ constraints, we show that the full set of solutions to systems of equations in a hyperbolic group remains  EDT0L.

  Our work combines the geometric results of Rips, Sela, \DG\ on the decidability of the existential theory of hyperbolic groups with the work of computer scientists including Plandowski, Je\.z, Diekert and others  on \PSPACE\ algorithms to solve equations in free monoids and groups using compression, and involves an intricate language-theoretic analysis.

\end{abstract}

\section{Introduction}

Let $G$ be a hyperbolic group with finite symmetric generating set $S$. In this paper we show that any system of equations in $G$ has solutions which, when written as shortlex representatives over $S$, 
admit a particularly simple description as formal languages, and moreover, this description can be given in very low space complexity. Our work combines the geometric results for determining the satisfiability of equations in hyperbolic groups of Rips, Sela, \DG\ \cite{DG,RS95}, with recent tools developed in theoretical computer science which give \PSPACE \ algorithms for solving equations in semigroups and groups \cite{CDE, DEijac,dgh01, DiekJezK,MR3571087, Jez2,Jez1}. 

The satisfiability of equations in torsion-free hyperbolic groups is decidable by the work of Rips and Sela \cite{RS95}, who reduced the problem  to solving equations in free groups, and then called on Makanin's algorithm \cite{mak83a} for free groups. 
Kufleitner proved  \PSPACE\ for decidability  in the torsion-free case \cite{DIP-1922}, without an explicit complexity bound, by following Rips-Sela and then using Plandowski's result  \cite{Plandowski}.  
\DG\  extended Rips and Sela's work to all
 hyperbolic groups (with torsion), by showing that it is sufficient to solve systems in virtually free groups, which they then reduced to systems of {\em twisted} equations in free groups \cite{DG}.

The first algorithmic description of {\em all} solutions to a given equation over a free group is due to Razborov \cite{MR1320290, MR755958}.
His description became known as a \emph{Makanin-Razborov (MR) diagram}, and this concept was then generalised to hyperbolic groups by Reinfeld and Weidmann \cite{RW14}, and to relatively hyperbolic groups by Groves \cite{MR2209374}. 
While in theory MR diagrams can be  used to algebraically produce the solutions of an equation via composition of group homomorphisms, it is infeasible to use this approach to directly obtain solutions as freely reduced  or shotlex representative words, as cancellations in the images of these homomorphisms cannot be controlled. Also, it is extremely complicated to explicitly produce a Makanin-Razborov diagram for a given equation even for free groups, and this has been done only in very few cases (\cite{MR2542213}). For the special case of quadratic equations, Grigorchuck and Lysionok gave efficient algorithms to obtain the solutions in \cite{GLquadratic}.

Here, we describe the combinatorial structure of solution sets of systems of equations: we show that they form an {\em EDT0L language} over the generating set $S$.
Roughly speaking, this means that there is a set $C\supseteq S$ and a {\em seed} word $c_0\in C^*$ so that every solution can be obtained by applying to $c_0$ a certain set of endomorphisms of the free monoid $C^*$. The set of endomorphisms is described by a finite labeled directed graph. (A simple EDT0L language can be seen in Example \ref{eg:AlexL}.) Our results are that we can algorithmically construct this graph, together with the set $C$ and $c_0$, and hence a finite description of the solution set, and moreover we can produce this description in 
 \PSPACE\ (see  \secref{sec:pspace} for space complexity definitions). The language-theoretic characterisation of solution sets had been open for a number of years even for free groups, and while this was settled by \cite{CDE}, it is remarkable that the EDT0L characterisation for free groups is so robust that it holds for the much bigger class of hyperbolic groups.
 
More specifically, we combine  Rips, Sela, \DG's approach with recent work of the authors with Diekert \cite{CDE,DEijac} to obtain the following results. An {\em inequation} is simply an expression using $\neq$ instead of $=$. The acronym {\em \qier} stands for {\em \Qier} and is defined in Subsection~\ref{subsec:qierDefn}.

\begin{theoremx}[Torsion-free]
\label{thm:IntroTorsionFree}
Let $G$ be a torsion-free hyperbolic group with finite symmetric generating set $S$. 
Let $\Phi$ be a system of equations and inequations
of size $n$ with constant size  effective \qier\  constraints  (see \cref{sec:notationSolns,defn:explicit} for  precise definitions). 
 Then \begin{enumerate}\item the set of all solutions, as tuples of shortlex geodesic words over $S$, 
is EDT0L, and the algorithm which on input $\Phi$ prints a description for the EDT0L grammar runs in $\NSPACE(n^2\log n)$.
 \item 
 it can be decided in $\NSPACE(n^2\log n)$ whether or not the solution set of $\Phi$ is empty, finite or infinite.
\end{enumerate}
\end{theoremx}

\begin{theoremx}[Torsion]
\label{thm:IntroTorsion}
Let $G$ be a  hyperbolic group with torsion, with finite symmetric generating set $S$. 
Let $\Phi$ be a system of equations and inequations
of size $n$ with constant size effective  \qier\  constraints  (see \cref{sec:notationSolns,defn:explicit} for  precise definitions). 
 Then \begin{enumerate}\item the set of all solutions, as tuples of shortlex geodesic words over $S$, 
is EDT0L, and the algorithm which on input $\Phi$ prints a description for the EDT0L grammar runs in $\NSPACE(n^4\log n)$.
 \item 
 it can be decided in $\NSPACE(n^4\log n)$ whether or not the solution set of $\Phi$ is empty, finite or infinite.
\end{enumerate}
\end{theoremx}

\Cref{thm:IntroTorsionFree,thm:IntroTorsion} follow from the more general \cref{thmTorsionFree,thmTorsion} stated later in the paper, and
have the following immediate consequence.
\begin{corollary}[Existential theory]\label{cor:existential}
The existential theory with constant size effective  \qier\ constraints can be decided in $\NSPACE(n^2\log n)$ for torsion-free hyperbolic groups and $\NSPACE(n^4\log n)$ for hyperbolic groups with torsion.  
\end{corollary}

In this paper we produce the solution sets not only in terms of shortlex representatives, but more generally in terms of \qgeods:
 for given constants, we can obtain the full set of solutions expressed as words belonging to some regular subset of quasigeodesics surjecting to the group, such as the set of all geodesics.
\Cref{table:results} 
summarises the different kinds of expressing solution sets, together with the corresponding language and space complexities that we are able to prove.

In the special case of free groups, we may want to produce \emph{all} words which represent solutions, and then such a set has a slightly higher complexity: ET0L instead of EDT0L (see  \secref{sec:EDT0L} for definitions).

\begin{restatable}[Full solutions in free groups]{corollary}{allFree}
\label{cor:allwords}
Let $G$ be a finitely generated free group with free basis  $A_+$,  and $A=A_+\cup \{x^{-1}\mid x\in A_+\}$ the {\em free basis generating set} for $G$.
Let $\Phi$ be a system of equations and inequations
of size $n$ 
with constant size 
rational constraints. 

Then the set of all solutions, as tuples of all words over $S$, 
is ET0L, and the algorithm which on input $\Phi$ prints a description for the ET0L grammar runs in  $\NSPACE(n\log n)$.
\end{restatable}

Note that the {\em word problem} for a group $G$ is the set of solutions to the one variable equation $X=1$, so \cref{cor:allwords} cannot be improved to EDT0L since it is known that the word problem for $F_2$ is not EDT0L (see {\cite[Proposition 26]{CEF}; \cite{MR0483746}). Also, it is suspected that word problems of hyperbolic groups are not ET0L unless the group is virtually free (in which case the word problem is deterministic context-free \cite{ms83}), 
so the requirement that words are \qgeods\ in all our results apart from \cref{cor:allwords}
most likely cannot be weakened in general.

\bigskip

\begin{table}[ht]
\begin{center}
\resizebox{\columnwidth}{!}{%
\rowcolors{1}{}{lightgray}
\begin{tabular}{r|r|r|r|r|l}
Class of groups &gen set &solutions as & language  & $\NSPACE$ &\\
  \hline
  Free & free basis & freely red words &  EDT0L & $n\log n$ &\cite{CDE}\\
Free &any & unique quasigeods  & EDT0L & $n\log n$ &   Cor.~\ref{cor:changeGset}\\
Free &any &  quasigeods  & ET0L & $n\log n$& Cor.~\ref{cor:changeGset}\\
Free &free basis & all words   & ET0L & $n\log n$& Cor.~\ref{cor:allwords}\\
Virt free & certain & certain quasigeods  & EDT0L & $n^2\log n$ & \cite{DEijac}\\
Virt free &any & unique quasigeods  & EDT0L & $n^2\log n$& Cor.~\ref{cor:changeGset}\\
Virt free &any & quasigeods & ET0L & $n^2\log n$& Cor.~\ref{cor:changeGset}\\
Torsion-free hyp &any & unique quasigeods & EDT0L & $n^2\log n$ &  Thm.~\ref{thmTorsionFree}\\
Torsion-free hyp &any & quasigeods  & ET0L & $n^2\log n$ &  Thm.~\ref{thmTorsionFree}\\
 Hyp with torsion&any & unique quasigeods  & EDT0L & $n^4\log n$ &   Thm.~\ref{thmTorsion}\\
 Hyp with torsion &any & quasigeods  & ET0L & $n^4\log n$&   Thm.~\ref{thmTorsion}\\
\end{tabular}%
}
\end{center}
\caption{Summary of results}
\label{table:results}
\end{table}

The above results partially extend to systems of equations and inequations subject to arbitrary \Qier\ constraints. 
By `partially extend' we mean that the language-theoretic complexity is preserved, but we can no longer guarantee that the description  
can be computed in (non-deterministic) polynomial space.
The issue here is that a Benois-type transition required to catch all solutions obeying the constraint does not preserve space complexity -- the automata construction blows up (see \cref{sec:DG9.4issue} for further discussion).
  Conceivably there may be another approach to get a polynomial bound on space complexity. For now we are able to state the following:

\begin{theoremx}[Systems with \Qier\ constraints]\label{thmCWithConstraints}
Let $G$ be a  hyperbolic group with or without torsion, with finite symmetric generating set $S$. 
Let $\Phi$ be a system of equations and inequations with arbitrary size effective  \qier\ constraints  (see \cref{sec:notationSolns,defn:explicit} for  precise definitions). 
Then the set of all solutions, as tuples of shortlex geodesic words over $S$, 
 is EDT0L. 
\end{theoremx}
\Cref{thmCWithConstraints} follows from the more specific statement in \cref{thmQIER} below.

EDT0L is a surprisingly low language complexity for solution sets in hyperbolic groups. EDT0L languages lie strictly in the class of indexed languages, and while containing all regular languages, they are incomparable to the context-free ones. See \cref{fig:EDTdiagram}.
Solution sets to systems of equations are content-sensitive, since there exists a Turing machine that can simply take an input tuple of words and substitute them into the equation to check whether they form a solution, needing just linear space. However, solution sets are not context-free in general, for example,  the equation $X=Y$ has solutions of the form $(w,w)$ which, if expressed as a language  $w\#w$ is a standard non-context-free example. 
As mentioned above, whether or not solution sets for free groups were even indexed languages was a long-standing open problem \cite{FerteMarinSenizerguesTocs14,GilPC,Jain} 
which was resolved by \cite{CDE}, and the present paper radically extends this to the much larger class of hyperbolic groups.

ET0L and EDT0L languages are playing an increasingly  useful role in group theory, not only in describing solution sets to equations in  groups \cite{CDE,DEijac,DiekJezK}, but more generally \cite{BEboundedLATA,BCEZ,CEF}.

\begin{figure}[ht]
    \centering
  \begin{tikzpicture}
  [node distance=2.5cm]
 \node (reg) {regular};
  \node (edt) [ right of =reg] {EDT0L};
  \node (cf) [ below of = edt, yshift=1cm] {context-free}; 
    \node (et) [ right of  = edt] {ET0L};
        \node (ind) [ right of = et] {indexed};
  \node (cs) [ right of = ind, xshift=5mm] {context-sensitive}; 
      
    \draw (reg) -- (edt);
        \draw (edt) -- (et);
           \draw (et) -- (ind);
    \draw (reg) -- (cf);
            \draw (cf) -- (et);
             \draw  (ind) -- (cs);
 \end{tikzpicture}

    \caption{Relationships between formal language classes (containment 
     from left to right).}\label{fig:EDTdiagram}
\end{figure}
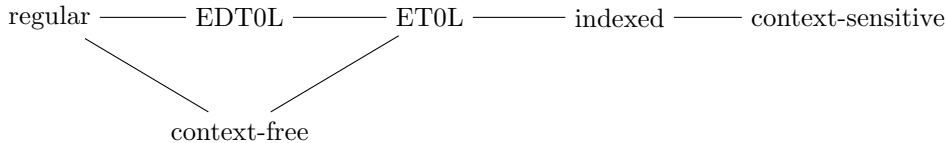

The paper is organised  as follows.
 In \Cref{sec:notationSolns,sec:EDT0L,sec:hyp-intro} we give the necessary background on equations, constraints, EDT0L languages, space complexity and hyperbolic groups. \Cref{sec:Copy} describes our key trick to produce solutions as EDT0L  rather than simply ET0L  languages.
 \Cref{sec:DG9.4issue} explains how we handle \qier\ constraints and complexity issues for them.
In \Cref{sec:torsionfree} we deal with the torsion-free case. We follow Rips and Sela's approach to solving systems of equations over torsion-free hyperbolic groups by reducing them to systems over free groups via {\em canonical representatives}, and we apply \cite{CDE} to show the language of all solutions is EDT0L. In \Cref{sec:torsion}  
we use \DG's reduction of systems over torsion hyperbolic groups to systems over virtually-free groups, via canonical representatives in an appropriate (subdivision of a) Rips complex, and we apply \cite{DEijac} together with intricate language operations, to show the language of all solutions is EDT0L. The torsion case is more involved because it requires keeping track of generating sets, \qgeod\ words (and paths in different graphs) and translations between these, to obtain the formal language description, and to show the $\NSPACE(n^4\log n)$ complexity. (See \cref{rmk:torsionHarder} for a brief explanation.)
\Cref{sec:constraints} extends the results  from the main part of the paper to systems with arbitrary \qier\ constraints. We explain why our techniques still ensure  EDT0L solutions but {\em a priori} not in \PSPACE. In \Cref{sec:further} we prove \cref{cor:allwords}.

An extended abstract of a {preliminary version of} this paper was presented at the conference ICALP 2019,  Patras (Greece),  8-12 July 2019
\cite{CiobanuEicalp2019}.

\section{Preliminaries -- Equations,  solution sets, rational constraints}
\label{sec:notationSolns}

A language $L\subseteq \Sigma^*$ is {\em regular} if there exists some (nondeterministic) finite state automaton over $\Sigma$ which accepts exactly the words in $L$. We abbreviate {\em nondeterministic finite automaton} to NFA.

Let $G$ be a fixed  group with finite symmetric generating set $S$, and $\pi:S^*\to G$ the natural projection map. 
A subset $R\subseteq G$ is  {\em rational} if $R=\pi(L)$ for some regular language $L\subseteq S^*$. 

 For a word $w \in S^*$ let $|w|_S$ denote the length of $w$, and for $g\in G$ let $\Abs{g}_S=
\min\{\abs{w}_S\mid w\in S^*,\pi(w)=g\}$ (the geodesic length of $g$ with respect to $S$).

\begin{definition}[System of equations]\label{defn:system}

Let $\mathcal X= \{X_1, \dots, X_{r}\}$  be a set of variables, $\mathcal{A}=\{a_1, \dots, a_k\} \subseteq G$ a set of constants, and $\phi_j(\mathcal X, \mathcal{A})\in (\mathcal X^{\pm 1}, \mathcal{A}^{\pm 1})^*$ a set of words, where  $1\leq j \leq s$, $k, r, s\geq 1$.
For each $X\in \mathcal X$, let $R_X$ be a rational subset of $G$.
Then
$$\Phi=\{\phi_j(\mathcal X, \mathcal{A})=1\}_{j=1}^h \cup \{\phi_j(\mathcal X, \mathcal{A})\neq1\}_{j=h+1}^s\cup\{R_X\mid X\in\mathcal X\}$$
is a \emph{system of equations and inequations  with rational constraints} over $G$.

 If each  word $\phi_j(\mathcal X, \mathcal{A})$ has length $l_j$ for all $j$, and each
  $R_X$ is given by an NFA with $|R_X|$ states, 
 the \emph{size} of the system $\Phi$, denoted $\Abs{\Phi}$, is defined using two parameters: 
 \[|\Phi|_1:=
 \sum_{j=1}^s l_j, \
 |\Phi|_2=\sum_{X\in \mathcal X}|R_X|, \text{ and } \Abs{\Phi}=\left(|\Phi|_1,|\Phi|_2\right).\]

A system is said to have {\em constant size rational constraints} if $|\Phi|_2$ is bounded by a constant (that might depend the fixed group and generating set). In this case we abuse notation and let the size be $\Abs{\Phi}=|\Phi|_1=n$.
  \end{definition}

\begin{definition}[Different types of solution sets]\label{solutiondef}\leavevmode

\begin{itemize}
\item[(i)] A tuple $(g_1,\dots, g_r)\in G^r$ \emph{solves} (or {\em is a solution of}) the system $\Phi$ if there exists a homomorphism $\sigma:F(\mathcal{X})\ast G \rightarrow G$ given by $\sigma(X_i)=g_i$ which fixes $G$,  satisfies $\sigma(X_i)\in R_{X_i}$ for each $1\leq i\leq r$ and 
$$\sigma(\phi_j(X_1, \dots, X_r, a_1, \dots, a_k))=1 $$ $$\text{and}\ \sigma(\phi_l(X_1, \dots, X_r, a_1, \dots, a_k))\neq1$$ 
for all $1\leq j\leq h$ and $h+1 \leq l \leq s$.
\item[(ii)] The {\em group element solutions} to $\Phi$ is the set $$\text{Sol}_G(\Phi)=\{(g_1,\dots, g_r) \in G^r \mid (g_1,\dots, g_r) \text{ solves } \Phi\}.$$

\item[(iii)] Let $\langT\subseteq S^*$ and $\#$ a symbol not in $S$. 
 The  {\em full set of $\langT$-solutions} is the set 
 $$\text{Sol}_{\langT,G}(\Phi)=\{w_1\#\dots\# w_r  \mid w_i \in \langT,  (\pi(w_1),\dots, \pi(w_r) ) \text{ solves  } \Phi\}.$$ 
\item[(iv)] Let $S_{\#r}=\{w_1\#\dots \#w_r\mid w_i\in S^*\}$ denote the set of all words over $S\cup\{\#\}$ which contain exactly $r-1$ $\#$ symbols.
A set
$$L=\{w_1\#\dots\# w_r\}\subseteq S_{\#r}$$
is  a {\em covering solution set} to $\Phi$ if  $ w_i \in S^*, 1\leq i \leq r$, and \[\{(\pi(w_1),\dots, \pi(w_r))\mid (w_1\#\dots\# w_r) \in L \}=\text{Sol}_G(\Phi).\] 
\end{itemize}
\end{definition}

\begin{example}
Let $G=\langle a, b, c, d \mid aba^{-1}b^{-1}cdc^{-1}d^{-1}=1\rangle$ be the surface group of genus $2$ with symmetric generating $S=\{a^{\pm 1},b^{\pm 1},c^{\pm 1},d^{\pm 1}\}$. Let $\mathcal X= \{X,Y\}$ be a set of variables, $\mathcal A=\{a, b, c, d\}$ a set of constants, and consider the equation 
$\phi(\{X,Y\}, \{a, b, c, d\})=abXcY$. 
Then the  system consisting of a single equation
\[\Phi=\left\{ abXcY=1\right\}\] has size $5$.
The pair $(a^{-1}b^{-1}, dc^{-1}d^{-1})$ is one group element solution of $\Phi$.

Now let $R_X=R_Y=R=\{g \in G \mid \Abs{g}_S \textrm{ is even}\}$. One can prove\footnote{By \Cref{prop:reg-hyp} below,  the language $L_1$ of all geodesics  is regular, which we can intersect with the regular language $L_2$ of all words of even length. Then $R=\pi(L_1\cap L_2)$.} 
that $R$ is a rational set in $G$. 
If \[\Phi'=\{abXcY=1\}\cup\{R_X, R_Y\},\] 
 then $\Phi'$  has no solutions: if $\sigma(X)$ and $\sigma(Y)$ are solutions of $\Phi'$ of even length then $\sigma(\phi(X,Y, \{a, b, c, d\}))=\sigma(a b X c Y)=ab\sigma(X)c\sigma(Y)$ has odd length and cannot represent the identity in $G$ because $G$ has only even length relators and only even length words can represent the identity in $G$.
\end{example}

\begin{example}
Let $G$ be a hyperbolic group with finite symmetric generating set $S$. 
If $\Phi$ consists of the single equation $X=1$ then $\text{Sol}_G(\Phi)=\{1\}$ and $\text{Sol}_{S^*,G}(\Phi)$ is the word problem of $G$ 
(believed to not be ET0L if $G$ is not virtually free).
\end{example}

\section{Preliminaries -- languages and space complexity} \label{sec:EDT0L}

An alphabet is a finite set. We use the notation $\mathscr P(A)$ for the power set (set of all subsets) of a set $A$, and $|A|$ for the size of  (number of elements in)  the set $A$.

\subsection{ET0L and EDT0L languages}

Let $C$ be an alphabet. 
A \emph{table} for $C$ is a finite subset of $C\times C^*$ which includes at least one element $(c,v)$ for each $c\in C$.
A table $t$ is {\em deterministic} if for each $c\in C$ there is exactly one $v\in C^*$ with  $(c,v)\in t$.
If $(c,v)$ is in some table $t$, we say that $(c,v)$ is a \emph{rule} for $c$. Applying a rule $(c,v)$ to  a letter $c$ means replacing $c$ by $v$.

If $t$ is a table and $u\in C^*$ then  we write 
$u\longrightarrow^t v$ to mean that $v$ is obtained by applying rules from $t$ to each letter of $u$. That is, $u=a_1\dots a_n$,  $a_i\in C$,  $v=v_1\dots v_n$, $v_i\in C^*$, and $(a_i,v_i)\in t$ for $1\leq i\leq n$. Note that when $t$ is deterministic then the word  $v$ obtained from $u$ by applying $t$ is unique. In this case we can write $v=t(u)$ instead of $u\longrightarrow^t v$.

If $H$ is a set of tables  and $r\in H^*$ then we write $u\longrightarrow^{r} v$ to mean that there is a sequence of words $u=v_0,v_1,\dots, v_n=v\in C^*$ 
such that $v_{i-1}\longrightarrow^{t_i} v_i$ for $1\leq i\leq n$ where $r=t_1\dots t_n$. If $A\subseteq H^*$ we write $u\longrightarrow^{\control} v$ if $u\longrightarrow^{r} v$ for some $r\in \control$.

\begin{definition}[\cite{Asveld}]\label{def:et0lasfeld}
	Let $\Sigma$ be an alphabet. We say that $L\subseteq \Sigma^*$ is an {\em ET0L} language if there is an alphabet $C$ with $\Sigma\subseteq C$, a finite set $H\subset \mathscr P(C\times C^*)$ of tables, 
	a regular language $\control \subseteq H^*$ and a fixed word $c_0\in C^*$ such that
\begin{displaymath}
	L = \{ w \in \Sigma^* \mid c_0 \longrightarrow^{\control} w\}. 
\end{displaymath}
In the case when every table $t \in H$ is deterministic, i.e. each $ h \in \control$ is in fact a homomorphism, we write $	L = \{ h(c_0) \in \Sigma^* \mid h\in \control \}$
 and say that $L$ is {\em EDT0L}.
The set ${\control}$ is called the {\em rational control}, the word $c_0$ the \emph{seed} and $C$ the {\em extended alphabet}\footnote{
The letters E,D,T,L stand for {\em extended, deterministic, table, Lindenmayer} respectively, and $0$ is  the number zero standing  for {\em $0$-interaction}.}.
\end{definition}

\begin{example}\label{eg:AlexL} \footnote{We thank Alex Levine for providing this example.}  
The language $L=\{a^{n^2} \mid n\in \mathbb{N}\}$ over the alphabet $\Sigma=\{a\}$ is EDT0L but not context-free. The extended alphabet is
     $ C = \{ s,  t,  u,  a\}$, seed word is $c_0=tsa$, 
     \[\begin{array}{lll}
\phi_1=\{( s, su)\}\\
\phi_2 =\{( t , at),(u,ua^2)\}\\
 \phi_3 =\{( s, \varepsilon), (t,\varepsilon) , (u,\varepsilon)\}\end{array}\] where we use here  the convention that $\phi_i$ fixes the elements in $C$ not explicitly specified, 
 $H=\{\phi_1, \phi_2, \phi_3: C \to C^*\}$, and $M$ the automaton in \cref{fig:AlexL}.
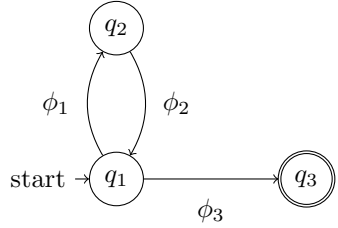
\begin{figure}[ht]
    \centering
   \begin{tikzpicture}  [node distance=2.5cm, scale=.4, every node/.style={circle}]

        \node[draw, initial] (q1)  {\(q_1\)};
        \node[draw, above of=q1,  yshift=-5mm] (q2) {\(q_2\)};
        \node[draw, right of=q1,accepting] (q3)  {\(q_3\)};

\draw[left,->,bend left] (q1) edge node {$\phi_1$} (q2);
\draw[right,->,bend left] (q2) edge node {$\phi_2$} (q1);
\draw[below,->] (q1) edge node {$\phi_3$} (q3);

      \end{tikzpicture}

    \caption{NFA for the rational control for the EDT0L grammar in \cref{eg:AlexL}.}\label{fig:AlexL}
\end{figure}
We have  $A=L(M) =(\phi_1 \phi_2)^\ast \phi_3$. 
One can check that $(\phi_1\phi_2)^i(tsa)=a^itsua^2ua^4u\dots ua^{2i}a$ which is sent to $a^{(i+1)^2}$ by applying $\phi_3$. It follows that 
the language of the EDT0L system is $\{a^{n^2} \ | \ n \in \mathbb{N}_+\}$.
\end{example}

\subsection{Space complexity for E(D)T0L}\label{sec:pspace}

Let  $s\colon \N\to\N$ be a function.
Recall an algorithm is said to run in \NSPACE$(s(n))$ if it can be performed by a non-deterministic Turing machine with a read-only input tape, a write-only output tape, and a read-write work tape, with the work tape restricted to using $\Oh(s(n))$ squares on input of size $n$.  We use the notation $L(\mathcal M)$ to denote the language  accepted by the automaton $\mathcal M$. The following definition formalises the idea of producing a description of some E(D)T0L language (such as the solution set of some system of equations) in \NSPACE$(s(n))$, where the language is the output of a computation with input (such as a system of equations) of size $n$.  We say an algorithm runs in \PSPACE\ if it runs in  \NSPACE$(s(n))$ for some polynomial function $s$.

\begin{remark}\label{hello}
Every $\NSPACE(s(n))$ algorithm (with $s(n) \in \OO(\log n)$)  can be simulated by a deterministic algorithm using at most working space $s(n)^2$ (Savitch's Theorem), and also by a deterministic Turing machine which uses a time bound in $2^{\Oh(s(n))}$, 
see \cite{pap94} for more details. Thus, every  $\PSPACE$ algorithm can be implemented such that it runs in deterministic singly exponential time  $2^{\text{poly}(n)}$.
\end{remark}

\begin{definition} Let 
 $\Sigma$ be a  (fixed) alphabet and  $s\colon \N\to\N$ a function.
  If there is an $\NSPACE(s(n))$ algorithm that on input $\Omega$ of size $n$ outputs (prints out)  a specification of an ET0L language $L_{\Omega}\subseteq \Sigma^*$, then we say that $L_{\Omega}$ is {\em ET0L in $\NSPACE(s(n))$}. 
 
 Here a specification of $L_{\Omega}$ consists of
  \begin{enumerate}\item
 an extended alphabet $C\supseteq \Sigma$, 
 \item a seed word $c_0\in C^*$, 
 \item a finite list 
of nodes of an NFA $\mathcal M$, labeled by some data, some possibly marked as initial and/or final, 
\item  a finite list $\{(u,v,h)\}$  of edges of 
$\mathcal M$ where $u,v$ are nodes and $h\in \mathscr P(C\times C^*)$ is a 
 table
\end{enumerate}
such that $L_{\Omega}=\{w\in \Sigma^*\mid c_0\to^{L(\mathcal M)} w\}$.

A language $L_{\Omega}$ is  {\em EDT0L 
 in
 $\NSPACE(s(n))$} if, in addition, every table $h$  labelling  an edge of $\mathcal M$ is deterministic.
 \end{definition}

 Note that the entire  print-out is not required to be in $O(s(n))$ space. 
 If  $|C|\in\Oh\left(\frac{s(n)}{\log n}\right)$ then we can  write out and store the entire extended alphabet as binary strings within our space bound, but in general this is just a convenience and not essential.
 
Previous results of the authors with Diekert can now be restated as follows.  
\begin{theorem}[{\cite[Theorem 2.1]{CDE}}] 
\label{thmCDE}
The set of all solutions  in a free group to a system of equations of size $n$, with  constant size  rational constraints,
as reduced words over the free generating set, is EDT0L in $\NSPACE(n\log n)$.
\end{theorem}
\begin{theorem}[{\cite[Theorem 45]{DEijac}}] 
\label{thmDE}
The set of all solutions in a virtually free group to a system  of equations  of size $n$, with constant size rational constraints,
as words in {\em standard normal forms}\footnote{a particular \qgeod\ normal form, see \cref{std-DE-quasigeod}.} over a certain finite generating set,  is EDT0L in $\NSPACE(n^2\log n)$. 
\end{theorem}

 The results in these papers are more general, in that they handle  rational constraints of certain non-constant sizes within the same space bound.
 We will return to the issue of  non-constant sized constraints in \cref{sec:constraints}.

\begin{remark}\label{rmk:space-for-table} In our  applications below we have $\Omega$ representing  some system of equations and inequations with (constant size) rational constraints, with $|\Omega|=n$, and we construct algorithms where the extended alphabet $C$ has size $|C|\in \Oh(n)$ in the torsion-free case and $|C|\in \Oh(n^2)$ in the torsion case. This means  we can write down the entire alphabet $C$ as binary strings within our space bounds.  Moreover, each element $(c,v)$ of any table we construct has $v$ of (fixed) bounded length, so we can write down entire tables within our space bounds.
 \end{remark}

\subsection{Closure properties}
It is well known 
 (see for example \cite[Theorem V.1.7]{MathTheoryLSystems})
  that the class of ET0L languages  is  closed under homomorphism, inverse homomorphism, finite union, and intersection with regular languages (a {\em full AFL}),  and  that EDT0L  is closed under all except  inverse homomorphism.
Here we show that the space complexity of an E(D)T0L language is preserved by these operations.  

Let us denote an E(D)T0L system by the $4$-tuple $(C, \Sigma, c_0,   \mathcal M)$ consisting of extended alphabet $C$,   alphabet  $\Sigma\subseteq C$, 
  seed word $c_0\in C^*$,  and  NFA $\mathcal M$ with edges labeled by tables in $\mathscr P(C\times C^*)$, so  that $L=\{w\in \Sigma^*\mid c_0\to^{L(\mathcal M)} w\}$ is the corresponding E(D)T0L language.

 \begin{proposition}\label{prop:closureET0L}
 Let $\Sigma, \Gamma$ be finite alphabets of fixed size, $\mathcal N$ an NFA of fixed size with $L(\mathcal N) \subseteq \Sigma^*$,  and  $\psi\colon \Sigma^*\to \Gamma^*$, $\varphi\colon \Gamma^*\to \Sigma^*$  homomorphisms. 
 Suppose there is a function $s\colon \N\to\N$ and  $\NSPACE(s(n))$ algorithms that on inputs $\Omega_1,\Omega_2$, each of size in $\Oh(n)$, prints out  specifications for  E(D)T0L languages
  $L_{\Omega_1}, L_{\Omega_2}\subseteq \Sigma^*$. (That is,  $L_{\Omega_1}, L_{\Omega_2}$ are E(D)T0L in \NSPACE$(s(n))$.) Then 
 \begin{enumerate}
     \item (union) $L_{\Omega_1}\cup L_{\Omega_2}$ 
 \item (homomorphism) $\psi(L_{\Omega_1})$ 
 \item (intersection with regular) $L_{\Omega_1}\cap L(\mathcal N)$
    \end{enumerate}
are  E(D)T0L in \NSPACE$(s(n))$ and 
 \begin{enumerate}
   \item[(4)] (inverse homomorphism) $\varphi^{-1}(L_{\Omega_1})$  \end{enumerate}
is  ET0L in \NSPACE$(s(n))$.

\end{proposition}

The proof is straightforward 
keeping track of complexity in the standard proofs, for example \cite[Theorem V.1.7]{MathTheoryLSystems} and \cite{AsveldChar, Culik}. We include a proof here for completeness.
\begin{proof}
To prove item
(1), assume the E(D)T0L language $L_{\Omega_i}$ is given by the $4$-tuple
$( C_i,  \Sigma, c_{i}, \mathcal M_i)$ which can be  constructed in \NSPACE$(s(n))$, for $i=1,2$. Let $p_i$ be the start state for $\mathcal M_i$.
Construct a new E(D)T0L system $(C_1\cup C_2, \Sigma, c_*, \mathcal M')$ where $c_*$ is a new seed  word consisting of a single letter, $t_i$ are two deterministic  tables defined by the rule $(c_*, c_{i})$, and $p_*$ the start state for $\mathcal M'$.
The automaton  $\mathcal M'$ is obtained by 
printing two edges $(p_*, p_i)$ labeled $t_i$ then printing all edges for $\mathcal M_1,\mathcal M_2$.
  Printing these can be done in space $s(n)$ given printing the data for $\mathcal M_1,\mathcal M_2$ can. Moreover since the two additional tables are deterministic,  the resulting language is EDT0L if both $L_{\Omega_i}$ are.

For the remaining items, assume the E(D)T0L language $L_{\Omega_1}$ is given by the $4$-tuple
$( C,  \Sigma, c_{0}, \mathcal M)$ which can be  constructed in \NSPACE$(s(n))$, with $p_0$ the start state for $\mathcal M$.

For item (2), 
print $\mathcal M$ with the following modifications. Print a new unique accept state $p_{\text{accept}}$ and for  each previous accept state of $\mathcal M$, print an edge from it to $p_{\text{accept}}$ labeled by the homomorphism (deterministic table) $\psi|_\Sigma$.
The resulting  language is $\psi(L_{\Omega_1})$ and the additional space required is in $\Oh(s(n))$ given that  printing the data for $\mathcal M$ requires this much space. If  $L_{\Omega_1}$ is EDT0L, since the table $\psi|_\Sigma$ added is deterministic, then so is $\psi(L_{\Omega_1})$.

To prove item 
(3), assume $\mathcal N$ has states $Q=\{q_0,\dots, q_r\}$,  with start state $q_0$. Without restriction (since $\mathcal N$ has constant size) we may assume $\mathcal N$ is deterministic and has a unique accept state $q_{\text{accept}}$. Construct a new E(D)T0L system $(C_*, \Sigma_*, [q_0,c_*,q_{\text{accept}}], \mathcal M')$    to accept $L_{\Omega_1}\cap L(\mathcal N)$ in $\NSPACE(s(n))$ as follows.

The states for $\mathcal M'$ are the states of  $\mathcal M$ plus a new start state $p_*$.
The new extended alphabet is \[C_*=\{[q_i,c, q_j]\mid c\in C, q_i,q_j\in Q\}\cup\{[q_0,c_*,q_{\text{accept}}]\}\] 
with seed word the single letter $[q_0,c_*,q_{\text{accept}}]$ where $c_*\not\in C$ is a new symbol. Define  
\[\Sigma_*=\{[q_i,a, q_j]\mid a\in \Sigma, q_i,q_j\in Q, (q_i,q_j,a) \text{ is an edge of } \mathcal N\}.\] 
For each  $(c,v)\in (C\cup\{c_*\})\times C^* $, define $\mathfrak r(c,v)$ to be the following set: if $v=a_1\dots a_n$ with $a_i\in C$, then 
 \[ \mathfrak r(c,v)=\left\{
 \left([q_{i_0},c,q_{i_n}], [q_{i_0},a_1,q_{i_1}][q_{i_1},a_2,q_{i_2}]\dots [q_{i_{n-1}},a_n,q_{i_n}]\right)
 \ \middle | \  q_{i_j}\in Q \right\}.\]

For each $(c_*,x)\in \mathfrak r(c_*,c_0)$
print an edge from $p_*$ to $p_0$ labeled by the deterministic table which sends $c_*$ to $x$ (and leaves remaining letters fixed).
For each edge $(p_s,p_t,h)$ printed by the  algorithm producing $\mathcal M$,   print a constant number
of new edges from $p_s$ to $p_t$ labeled by tables obtained from $h$ by replacing each element $(c,v)\in h$ by some choice from $\mathfrak r(c,v)$.
Since $\mathcal N$ is constant size, each edge can be printed in  $\Oh(s(n))$ space, that is, the same space needed to print edges for $\mathcal M$. Each table printed is deterministic if the original $h$ was.

The language $K$ of this system is the set of all strings over $\Sigma_*$ of the form \[[q_{0},a_1,q_{i_1}][q_{i_1},a_2,q_{i_2}]\dots [q_{i_{n-1}},a_n,q_{\text{accept}}]\] where $a_1\dots a_n\in L_\Omega$, $q_0, q_{\text{accept}}$ are the unique start and accept states of  $\mathcal N$, and $q_{i_j}$ are any possible choice of states of $\mathcal N$.  
Finally, define the homomorphism $\tau:\Sigma_*\to \Sigma$ by $\tau([q_{i},a,q_{j}])=a$, then  by construction $\tau(K)=L_{\Omega_1}\cap L(\mathcal N)$, and  by item (1) $\tau(K)$ is E(D)T0L in $\NSPACE(s(n))$.

To prove item 
(4), we have $\varphi\colon \Gamma^*\to \Sigma^*$ and $L_{\Omega_1}$  given by the $4$-tuple
$( C,  \Sigma, c_{0}, \mathcal M)$. Assume without loss of generality that $\Gamma\cap C=\emptyset$. Let $K\subseteq (\Gamma\cup\Sigma)^*$  defined by 
\[K=\left\{y\in (\Gamma\cup\Sigma)^*\mid y=z_0x_1z_1\dots x_kz_k, x_1\dots x_k\in L_{\Omega_1}, x_i\in \Sigma, z_i\in \Gamma^*\right\}\] 
be a ``padded" copy of $L_{\Omega_1}$. Define a new extended alphabet $C'=C\cup \Gamma$, and define a non-deterministic table 
$h_0=\{(a,xay)\mid a\in \Sigma, x,y\in \Gamma\cup\{\epsilon\}\}$ (and fixes any $a\in C'\setminus \Sigma$). Each table labelling an edge in $\mathcal M$ can be viewed as a table in $\mathscr P(C'\times (C')^*)$ (again by convention tables are the identity on letters not in $\Sigma$).
Modify $\mathcal M$ by adding loops labelled by $h_0$ to each accept state to obtain $\mathcal M'$. The resulting system $(C', \Gamma\cup\Sigma, c_0, \mathcal M')$ is ET0L in the same space complexity as the initial system since the modifications depend only on  $\Gamma$ (which is fixed size) and $C$, and the language of this system in $K$.

Now consider 
 the  regular language $S=\{\varphi(y_1)y_1\dots \varphi(y_n)y_n\mid n\geq 1, y_i\in \Gamma\}$. Then $\Gamma$ is of fixed size (not part of the input); also, the automaton accepting $S$ consists of state $q_0$ which is both the start and unique accept state, and paths starting and ending at $q_0$ labeled by $\varphi(y_i)y_i$, so is of fixed constant size. Thus
 by item (3) $S\cap K$ is ET0L in $\NSPACE(s(n))$.
Finally, define a homomorphism $\tau: (\Gamma\cup \Sigma)^*\to \Sigma$ by $\tau(a)=a$ if $a\in \Gamma$ and $\tau(a)=\epsilon$ if $a\in \Sigma$. Then by item (2) 
  $\tau(K\cap S)$ is  ET0L in $\NSPACE(s(n))$, and by construction $\tau(K\cap S)=\phi^{-1}(L_{\Omega})$. \end{proof}

\begin{notation}\label{notation:Tr}  Let $\Sigma$ be an alphabet with $\#\not\in\Sigma$.
Let $\langT\subseteq \Sigma^*$ and $r\in \N$. Define 
$\langT_{\#r}=\{u_1\#\dots \#u_r\mid u_i\in \langT, 1\leq i\leq r\}$. (Note that this agrees with the usage in \cref{solutiondef} item (iv).)
\end{notation}

 \begin{proposition}[Projection onto a factor]\label{prop:projection}
 Let $\Sigma$ be an alphabet with $\#\not\in\Sigma$, and  $1\leq k\leq \ell$.
Suppose for each input $\Omega$ there is a language $L_{\Omega} \subseteq (\Sigma\cup\{\#\})^*$ and an  integer $r_\Omega\geq 0$ so that 
$L_{\Omega}\subseteq (\Sigma^*)_{\#r_\Omega}$. In that case  we can 
 define \[L_{\Omega,k,\ell}=\{u_k\#\dots\#u_\ell\mid u_1\#\dots \#u_k\#\dots\#u_\ell\#\dots \#u_{r_\Omega}\in L_{\Omega}\}\]  if $\ell\leq r_\Omega$ and $L_{\Omega,k,\ell}=\emptyset$ otherwise.
 Then
  if  $L_{\Omega}$ is E(D)T0L in \NSPACE$(s(n))$, then
 $L_{\Omega,k,\ell}$  is E(D)T0L  in \NSPACE$(s(n))$.  \end{proposition}

\begin{proof}
Assume the E(D)T0L language $L_{\Omega}$ is given by the $4$-tuple
$( C,  \Sigma\cup\{\#\}, c_0, \mathcal M)$ which can be  constructed in \NSPACE$(s(n))$.
Construct a new E(D)T0L system $( C_*,  \Sigma\cup\{\#\}, [1,c_*, r_\Omega], \mathcal M')$ as follows.
The new alphabet is
\[C_*=\{[i,c,j], [1,c_*, r_\Omega]\mid c\in C, 1\leq i\leq j\leq r_\Omega\}\cup \Sigma\cup\{\#\},\] where $[1,c_*,r_\Omega]$ is a new seed letter, $c_*\not\in C$. The notation $[i,c,j]$ is intended to indicate that the word  in $(\Sigma\cup\{\#\})^*$ eventually produced in $L_{\Omega}$  by  following some tables starting with $c$ will be $v_i\#u_{i+1}\#\dots u_{j-1}\#v_j$, where $v_i$ is a suffix of $u_i$ and $v_j$ is a prefix of $u_j$.

 For each  $(c,v)\in (C\cup\{c_*\})\times C^* $, define $\mathfrak r(c,v)$ to be the following set: if $v=a_1\dots a_n$ with $a_i\in C$, then  $\mathfrak r(c,v)=$
 \[ \left\{([i,c,j],[i,a_1,s_1][i_1,a_2,i_2]\dots [s_{n-1},a_n,j]\mid 1\leq  i\leq s_1\leq \dots \leq s_{n-1}\leq j\leq r_\Omega\right\}.\]

Let $p_0$ be the start node of $\mathcal M$.
Make a new start node $p_*$, and for each element  $([1,c_*, r_\Omega],x)\in \mathfrak r(c_*,c_0)$ print an edge from $p_*$ to $p_0$ labeled by the table which sends $[1,c_*, r_\Omega]$ to $x$. 
For each edge in $\mathcal M$ labeled by table $h$, print all possible edges labeled by tables obtained from $h$ by replacing each rule $(c,v)$ in $h$ by some choice of element of $\mathfrak r(c,v)$. Note that the space required to print each edge is $\Oh(s(n))$.
For each final state of $\mathcal M$, print an edge to a new final state $q_*$ labeled by the homomorphism which for all $k\leq i\leq \ell,  k\leq j< \ell, c\in \Sigma$  sends $[i,c,i]$ to $c$  and $[j,\#,j+1]$ to $\#$ (and is constant on all other letters).

The resulting language is exactly the set of factors of $L_\Omega$ as required, the space to print $C_*$ and $\mathcal M'$ is $\Oh(s(n))$, and  all tables labelling edges in $\mathcal M'$ are deterministic whenever all tables were in $\mathcal M$.
 \end{proof}

\section{Preliminaries -- Hyperbolic groups}\label{sec:hyp-intro}

Recall that the {\em Cayley graph} for a group $G$ with respect to a finite symmetric generating set $S$ is the directed graph $\Gamma(G,S)$ with vertices labeled by $g\in G$ and a directed edge $(g,h)$ labeled by $s\in S$ whenever $h=_Ggs$.
Let $\ell_{\Gamma(G,S)}(p)$, $i(p)$ and $f(p)$ be the length, initial and terminal vertices of a path $p$ in the Cayley graph, respectively. 
A path  $p$ is {\it geodesic} if 
$\ell_{\Gamma(G,S)}(p)$ is minimal among the lengths of all paths $q$ with the same endpoints. 
If $x,y$ are two points in  $\Gamma(G,S)$, we define $d_S(x,y)$ to be the length of a shortest path from $x$ to $y$ in  $\Gamma(G,S)$.
We use $d(x,y)$ and $\ell(p)$ if the group $G$ and set $S$ are clear from the context.

 \begin{definition}[$\delta$-hyperbolic group (Gromov)]
 Let $G$ be a group with finite symmetric generating set $S$, and let $\delta\geq 0$ be a fixed real number. 
 If  $p,q,r$ are geodesic paths in $\Gamma(G,S)$ with $f(p)=i(q),f(q)=i(r),f(r)=i(p)$, we call $[p,q,r]$ a 
{\em geodesic triangle}. 
A geodesic triangle is {\em $\delta$-slim} if $p$ is contained in a $\delta$-neighbourhood of $q\cup r$, that is, every point on one side of the triangle is within $\delta$ of some point on one of the other sides.
We say $(G,S)$ is {\em $\delta$-hyperbolic} if every geodesic triangle in $\Gamma(G,S)$ is $\delta$-slim. We say $(G,S)$ is {\em hyperbolic} if it is $\delta$-hyperbolic for some $\delta\geq 0$. 
 \end{definition}

\begin{lemma}[{Dehn presentation, \cite[Theorems 2.12 and 2.16]{MSRINotes}}] 
 \label{lem:Dehn}
 $G$ is hyperbolic if and only if 
 there is a finite list of pairs of words $(u_i,v_i)\in S^*\times S^*$ with $|u_i|>|v_i|$ and $u_i=_G v_i$ such that the following holds: if $w\in S^*$ is equal to the identity of $G$ then it contains some $u_i$ as a subword.
 
 \end{lemma} 
 This gives  an algorithm to decide whether or not a word $w\in S^*$ is equal to the identity: while $\abs{w}_S>0$, look for some $u_i$ subword.
 If there is none, then $w\neq_G 1$. Else replace $u_i$ by $v_i$ (which is shorter). This procedure is called {\em Dehn's algorithm}. The following is immediate.

  \begin{lemma}
 Dehn's algorithm runs in 
 linear space.\end{lemma}

\begin{definition}[Quasigeodesic]\label{defn:qg}
Let $(X,d)$ be a metric space.
For $\lambda \geq 1, \mu\geq 0$ real numbers, a path $p$ in $(X,d)$ is an $(X,\lambda, \mu)$-{\em \qgeod} 
if for any subpath $q$ of $p$ we have 
$\ell(q)\leq \lambda d(i(q),f(q))+\mu$. Let $Q_{ X,\lambda, \mu}$ 
be the set of all $(X,\lambda, \mu)$-\qgeods. 
 
 When $(X,d)$ is the Cayley graph of a group $G$ with respect to a generating set $S$, we use the notation $(G,S,\lambda, \mu)$-{\em \qgeod},
 and 
 denote by \[Q_{ G,S,\lambda, \mu}=Q_{ \Gamma(G,S),\lambda, \mu}\cap S^*\] the set of all $(\lambda, \mu)$-\qgeod \ paths between vertices of the Cayley graph, or equivalently, paths corresponding to words over $S$.
\end{definition}

 \begin{lemma}[{Change of generating set is a quasi-isometry; folklore}]
 \label{change_of_gset}
 Let $S_1, S_2$ be two finite symmetric generating sets for a group $G$, $\pi_i:S^*_i \mapsto G$ the natural projection maps, and $\psi: S_1^*\to S_2^*$ the monoid morphism which satisfies $\pi_1(s)=\pi_2(\psi(s))$ for all $s \in S_1$.
Then for each $\lambda\geq 1, \mu \geq 0$ there exist  $\lambda',\mu'$ so that if $p$ is a $(G,S_1, \lambda,\mu)$-\qgeod, then $\psi(p)$ is a  $(G,S_2, \lambda',\mu')$-\qgeod.
 \end{lemma}

  Throughout this article, we assume $G$ is a fixed hyperbolic group with finite  generating set $S$ which we treat as a constant for complexity purposes. We also assume we are given  the constant $\delta$, the finite list of pairs $(u_i,v_i)$ for Dehn's algorithm, and any other constants depending only on the group, for example the constants $\lambda_G,\mu_G$ in \cref{canrep_qg} below. 
   
 \medskip

\subsection{Languages in hyperbolic groups}
 
 Let $\Lambda$ be an alphabet and $\$$ a symbol not in $\Lambda$. 
 An {\em asynchronous $2$-tape automaton} is a finite state automaton $\mathcal M$ with alphabet
  \[\mathcal X=\left\{{s\choose t},  {s\choose \$}, {\$\choose s}   \ \middle | \  s,t\in \Lambda \right\}.\] 
  The 2-tape automaton is {\em synchronous} if it only accepts words \[{x_1\choose y_1}\dots  {x_n\choose y_n}\] with either $x_1\dots x_n\in \Lambda^*$ and $y_1\dots y_n\in \Lambda^*\{\$\}^*$; or 
  $x_1\dots x_n\in \Lambda^*\{\$\}^*$ and $y_1\dots y_n\in \Lambda^*$.
  
   Define homomorphisms $\xi_1,\xi_2\colon \mathcal X  \to (\Lambda \cup \Lambda)^*$ by 
 \[ \xi_1\left({x\choose y}\right)=\left\{\begin{array}{llll} x & & x\in \Lambda\\1 & \ \ & x=\$ \end{array}\right. \text{ and  }  \ \  \xi_2\left({x\choose y}\right)=\left\{\begin{array}{llll} y & & y\in \Lambda \\1 & \ \ & y=\$\end{array}\right..\]  
  We say the pair  $(u,v)\in \Lambda^2$ is accepted by the (synchronous or asynchronous) 2-tape automaton if there is a word \[{x_1\choose y_1}\dots  {x_n\choose y_n}\] accepted by $\mathcal M$ such that $u=\xi_1(x_1\dots x_n)$ and $v=\xi(y_1\dots y_n)$.
See \cite[Definition 1.4.5 and \S{7.1}]{WordProc}  and \cite[\S{2}]{HoltRees} for equivalent formulations.

 It is standard practice for finitely generated groups to fix an order on the generators, order which can be extended to (a lexicographic order for) all words over that generating set, and then for each group element to choose a \emph{shortlex representative}; that is, choose the smallest word, in the lexicographic order, among all the words representing the group element.

\begin{proposition}[See \cite{WordProc,HoltRees}]\label{prop:reg-hyp}

Let $G$ be a fixed hyperbolic group with finite  generating set $S$, and $\lambda\geq 1, \mu\geq 0$ constants with $\lambda\in\mathbb Q$.
Then the following sets are regular languages.
\begin{enumerate}
\item
The set of all geodesics over $S$.
\item
The set of all shortlex geodesics over $S$. 
\item The set of \qgeods \
$Q_{G,S,\lambda, \mu}\subseteq S^*$. 
\end{enumerate}
Furthermore, the set of all pairs of words $(u,v)\in Q_{G,S,\lambda, \mu}^2$ such that $u=_Gv$ is accepted by an asynchronous 2-tape automaton.
\end{proposition}
Note that Holt and Rees \cite{HoltRees} show more: if $\lambda\not\in \mathbb Q$ then the set $Q_{G,S,\lambda, \mu}$  is never regular, and in the special case that $(\lambda,\mu)$ are ``exact" then the 2-tape automaton is synchronous. We do not need these sharper results here.

\subsection{Quasi-isometrically embedded rational constraints}\label{subsec:qierDefn}
Let $G$ be a group generated by a symmetric set $S$, and let $\pi:S^* \rightarrow G$ be the natural projection as before. Recall that for an element $g \in G$ we denote the geodesic length of $g$ with respect to the word metric for $S$ by $\Abs{g}_S$. 

\begin{definition}[see \cite{DG}]\label{defn:qier}
(1) A regular language $\widetilde{R} \subseteq S^*$ is $(\lambda,\mu)$-\emph{quasi-isometrically embedded}  in $G$ if for any 
$w \in \widetilde{R}$
$$\Abs{\pi(w)}_S \geq \frac{1}{\lambda}\abs{w}_S-\mu.$$
(2) A rational set $R \subseteq G$ is \emph{quasi-isometrically embeddable}  in $G$ if there exist real numbers $\lambda \geq 1$ and $\mu \geq 0$  and a $(\lambda,\mu)$-quasi-isometrically embedded regular language $\widetilde{R}\subseteq S^*$ such that $\pi(\widetilde{R})=R.$ In this case we say that $R$ is {\em \Qier}, abbreviated as {\em \qier}.
\end{definition}

For the purposes of our complexity and computability results, we also define the following.
\begin{definition}[Effective and explicit]
\label{defn:explicit}
A rational set $R$ is said to be {\em effective \Qier}  (abbreviated  as {\em effective \qier}) if we are given an NFA $A$ such that there exist constants $\lambda\geq 1,\mu\geq 0$ with  $L(A)\subseteq Q_{G,S,\lambda,\mu}$ and $\pi(L(A))=R.$
The set $R$  is said to be {\em explicit qier} if, in addition, the constants $\lambda\geq 1,\mu\geq 0$ are also  given.
\end{definition}

In \cite{DG}, decidability statements use the notion of effective \qier\ as defined here. Moreover, \DG\ show that given an effective \qier\ set, there is an algorithm to compute constants $\lambda, \mu$ such that the set is explicit \qier\ with respect to these constants,
see \cref{sec:DG9.4issue} below.
\smallskip

\subsection{The Rips complex}

For any metric space $X=(X,d)$ and constant $r$, the {\em Rips complex} with parameter $r$ is the simplicial complex whose
vertices are the points of $X$ and whose simplices are the finite subsets of $X$ whose diameter is at
most $r$. More specifically, we will need the Rips complex which is based on the Cayley graph of a hyperbolic group, as follows.

\begin{definition}[Rips Complex]\label{RipsComplex}
Let $G$ be a hyperbolic group with finite generating set $S$. For a fixed constant $r$, the Rips complex $\mathcal{P}_r(G)$ is a simplicial complex defined as the collection of sets with diameter $\leq r$ (with respect to the $d_S$-distance in $\Gamma(G,S)$):
$$\mathcal{P}_r(G)=\left\{Y\subset G   \ \middle | \  Y\neq \emptyset, \textrm{ \ diam}_S(Y) \leq r\right\},$$
where each set $Y \in \mathcal{P}_r(G)$ of cardinality $k+1$ is identified with a $k$-simplex whose vertex set is $Y$.
\end{definition}

The Rips complex of a hyperbolic group has several important properties, relevant to this paper. The group $G$ acts properly discontinuously on the Rips complex $\mathcal{P}_r(G)$, the quotient $\mathcal{P}_r(G)/G$ is compact, and $\mathcal{P}_r(G)$ is contractible. In \cref{sec:torsion} we will work with the barycentric subdivision of $\mathcal{P}_r(G)$, which we recall next.

\begin{definition}[Barycentric subdivision]
\leavevmode
Let $\sigma$ be a simplicial complex. For a simplex $\tau=\{v_0, v_1, \dots, v_q\} \in \sigma$ denote by $b_{\tau}$ its \emph{barycentre}, and 
for two simplices $\alpha, \beta$ in $\sigma$ write $\alpha<\beta$ to denote that $\alpha$ is a face of $\beta$. 

The \emph{barycentric subdivision} $B_{\sigma}$ of a simplicial complex $\sigma$ is the collection of all simplices whose vertices are $b_{\sigma_0}, \dots, b_{\sigma_r}$ for some sequence $\sigma_0 < \dots <\sigma_r$ in $\sigma$. 
\end{definition}

The set of vertices of $B_{\sigma}$ is the set of all barycentres of simplices of $\sigma$, $B_{\sigma}$ has the same dimension as $\sigma$, and any vertex in $B_{\sigma} \setminus \sigma$ is connected to a vertex in $\sigma$.

\section{Doubling and copying}\label{sec:Copy}

In computing the full solution set to equations as shortlex geodesic words, we will need to take inverse homomorphism.
Even though in general the image under an inverse homomorphism of an EDT0L language is just ET0L, because of the special structure of solution sets we can apply the {\em  Copying Lemma} of Ehrenfeucht and  Rozenberg \cite{EhrenRozenInverseHomomEDT0L}
 to show the statement in \cref{prop:shortlex-hyp}.
This is indeed a  {\em trick} -- without it we would only be able to state our main structural results (solutions as shortlex geodesics) at ET0L languages, whereas EDT0L is a much smaller class and hence we have stronger statements.

\begin{lemma}[{Copying Lemma, \cite[Theorem 1]{EhrenRozenInverseHomomEDT0L}}]
\label{lem:copyER}
Let $\Sigma_1,\Sigma_2$ be two finite disjoint alphabets, $K_1\subseteq \Sigma_1^*$ and $K_2\subseteq \Sigma_2^*$.
Let $f$ be a bijective function from $K_1$ onto $K_2$. Let $K=\{wf(w)\mid w\in K_1\}$. If $K$ is ET0L, then $K,K_1,K_2$ are each EDT0L.

\end{lemma}

\begin{lemma}[Copying in \NSPACE]\label{lem:copyME} 
Let 
  $s\colon \N\to\N$ be a function. Let $\Sigma_1,\Sigma_2$ be two finite disjoint alphabets, and $f\colon\Sigma_1\to \Sigma_2$ a  bijection which extends to a monoid homomorphism $f\colon\Sigma_1^*\to \Sigma_2^*$ of the same name.
Suppose that on input $\Omega$ languages $(K_\Omega)_1\subseteq \Sigma_1^*$ and  
\[K_\Omega=\{wf(w)\mid w\in (K_\Omega)_1\}\] are produced.
If $K_\Omega$ is ET0L in \NSPACE$(s(n))$, then $K_\Omega, (K_\Omega)_1$ (and $f((K_\Omega)_1)$) are each EDT0L in \NSPACE$(s(n))$.

\end{lemma}

\begin{proof} 
Following Steps 1--4 in the proof  of  \cite[Theorem 1]{EhrenRozenInverseHomomEDT0L}, 
 we see that each nondeterministic table in the grammar for $K$ is replaced by a finite number (independent of $\Omega$) of deterministic tables, with symbols superscripted by $(1),(2),(m)$, $(m:1)$, and $(m:2)$, where the letter $m$ stands for ``middle" (not an integer). It follows that all tables  for the EDT0L grammars for $K_\Omega, (K_\Omega)_1$ and $f((K_\Omega)_1)$ can be  printed  in $\Oh(s(n))$ space. 
\end{proof}

We will also make use of the following  fact. 

\begin{lemma}[Doubling  in \NSPACE]\label{lem:doubleME} 
Let 
  $s\colon \N\to\N$ be a function. Let $\Sigma_1,\Sigma_2$ be two finite disjoint alphabets, and $f\colon\Sigma_1\to \Sigma_2$ a  bijection which extends to a monoid homomorphism $f\colon\Sigma_1^*\to \Sigma_2^*$ of the same name.
If $L_\Omega\subseteq \Sigma_1^*$ is EDT0L in \NSPACE$(s(n))$ then $(L_\Omega)_D=\{wf(w)\mid w\in L_\Omega\}$ is EDT0L in  \NSPACE$(s(n))$.
\end{lemma}

 \begin{proof}
 Modify the grammar for $L_\Omega$ as follows.
Replace the seed word  $c_0$ by $c_0f(c_0)$, and for each rule $(a,u)$ in a table, add  $(f(a), f(u))$ to the table. These modifications are clearly in the same space bound and tables remain deterministic since alphabets are disjoint.
 \end{proof}
 
 Note that if the language $L_\Omega$  in \cref{lem:doubleME}  were ET0L  and not EDT0L, then the proof no longer works, since tables in the ET0L grammar could make different substitutions to letters in the prefix than in the suffix. For this reason\footnote{and more obviously, because we know EDT0L is a proper subclass of ET0L \cite{MR413615Copying}} we cannot simply double any ET0L language and then apply \cref{lem:copyER} to prove that it is  EDT0L. 
 We remark that  \cite[Theorem 3.3]{MR413615Copying} is a slightly stronger statement of copying, but we don't use this here.  
 
Here is our  key technical result for languages of words over hyperbolic groups.
 It will show how we can  build, from a given covering solution set of quasigeodesic words which is ET0L, an ET0L language consisting of all words in some regular language $\langT$ which correspond to one of the covering solutions, without increasing the amount of space required. In the case that $\langT$ is  
 a set of normal forms for the group (unique representative for each element), the Copying Lemma (\cref{lem:copyME}) ensures that the resulting language is in fact EDT0L.

 First we fix some notation to be used throughout the paper.

\begin{notation}\label{notation:doubleGset}
Let $G$ be a hyperbolic group with finite symmetric  generating set $S$ and natural projection map $\pi:S^*\to G$. 
  Let $S_\dag=\{x_\dag\mid x\in S\}$ be a disjoint copy of the alphabet $S$ where every letter is marked with a subscript $\dag$, and define a bijective function $f\colon S\to S_\dag$ by $f(x)=x_\dag$, which we can extend to the free monoid homomorphism $f\colon S^*\to S_\dag^*$. 
Then $\pi_\dag=\pi\circ f^{-1}$  is a map from $S_\dag^*$ to $G$, and we can formally consider $S_\dag$ as a generating set for $G$ with projection map $\pi_\dag$.
\end{notation}

\begin{proposition}[Covering  to full  sets]\label{prop:shortlex-hyp}
Let $G$ be a hyperbolic group with finite symmetric  generating set $S$. For $\#$ a symbol not in $S$, extend the bijection $f\colon S\to S_\dag$ defined in \cref{notation:doubleGset} to include  $f(\#)=\#$.

Let  $\lambda, \lambda',\mu,\mu'\in \mathbb R$ be given fixed constants with $\lambda, \lambda'\geq 1$, $\mu,\mu'\geq 0$ and $\lambda'\in\mathbb Q$. Let  $\mathcal T\subseteq Q_{G,S,\lambda', \mu'}$ be a fixed regular language of $(G,S,\lambda', \mu')$-quasigeodesics, and $r$ be a fixed positive integer.

Suppose some language \[L_{\text{\em cover}}\subseteq \{u_1\#\dots \# u_rf(v_1\#\dots \# v_r)\mid u_i,v_i\in Q_{G,S,\lambda, \mu},  u_i=_Gv_i, 1\leq i\leq r\}\]
  is ET0L in $\NSPACE(s(n))$.  Then
\begin{enumerate}
     \item\label{item:surjectSet} 
     if  the projection $\pi\colon \mathcal T\to G$ is a surjection,
then
  \[L_{\mathcal T}=\{w_1\#\dots \#w_r \mid \exists u_1\#\dots\#u_r f(v_1\#\dots\#v_r) \in L_{\text{\em cover}}, w_i=_Gu_i, w_i\in \mathcal T \}\]    is ET0L  in $\NSPACE(s(n))$.
  \item \label{item:shortlex}
  if the projection $\pi\colon \mathcal T\to G$ is a bijection,
then 
 \[L_{\mathcal T}=\{w_1\#\dots \#w_r \mid \exists u_1\#\dots\#u_r  f(v_1\#\dots\#v_r) \in L_{\text{\em cover}}, w_i=_Gu_i, w_i\in \mathcal T \}\]    is EDT0L  in $\NSPACE(s(n))$.
\end{enumerate}
  
  Moreover, in both cases  $L_{\mathcal T}$  is finite (resp. empty) if and only if $L_{\text{\em cover}}$ is finite (resp. empty).
  \end{proposition}

Note the requirement that $\lambda'\in \mathbb Q$, this is because  we make use of the fact that the set of all \qgeods\ in a hyperbolic group gives an asynchronous automatic structure which only holds when the constant $\lambda'$ is rational by \cite{HoltRees}. The value of $\lambda'$ is independent of $\lambda$ (can be larger or smaller).

\begin{proof} Let $\Sigma_1=S\cup\{\#\}$, and let $\Sigma_2=S_\dag\cup\{\#_\dag\}$ be a disjoint copy of $\Sigma_1$. 
Choose $\lambda_*,\mu_*\geq 0$ so that $\lambda_*\in \mathbb Q$, $\lambda_*\geq \max\{\lambda,\lambda'\}$ and  $\mu_*\geq \max\{\mu,\mu'\}$. Recall that 
$\mathcal T_{\#r}=\{u_1\#\dots \#u_r\mid u_i\in \mathcal T\}$.
Observe that if $\langT$ is regular and of fixed constant size, then so is $\langT_{\#r}$: take $r$ disjoint copies of the automaton accepting $\langT$ and join  by $\#$ transitions each accept state of the $i$th copy to the start state of the $(i+1)$st copy.

Let $\$$ denote a `padding symbol' that is distinct from $\Sigma_1\cup \Sigma_2$, and let 
 \[\mathcal P=\left\{{s\choose t},  {s\choose \$}, {\$\choose s}, {\#\choose \#}, {s_\dag\choose t_\dag},{s_\dag\choose \$}, {\$\choose s_\dag},   {\#_\dag\choose \#_\dag}  \ \middle | \ s,t\in S, s_\dag,t_\dag\in \Sigma_2 \right\}.\]
 Define a homomorphism $\psi\colon \mathcal P^*\to (\Sigma_1\cup \Sigma_2)^*$ by 
 \[ \psi\left({x\choose y}\right)=\left\{\begin{array}{llll} 1 & \ \ & x=\$\\x & & x\in \Sigma_1\cup \Sigma_2\end{array}\right.\] 
 Then $\psi^{-1}(L_{\textrm{cover}})$ is ET0L in \NSPACE$(s(n))$ by \cref{prop:closureET0L}(4), and consists of strings of letters which we can view as two parallel strings: the top string is a word from $L_{\textrm{cover}}$ with $\$$ symbols inserted, and the bottom string can be any word in $\Sigma_1\cup \Sigma_2\cup\{\$\}$ with $\#,\#'$ occurring  in exactly the same positions as the top string, and  no two $\$$ symbols in the same position top and bottom.
 
 Let ${\mathcal M}$ be the asynchronous 2-tape automaton which accepts all pairs $(u,v)\in Q_{G,S,\lambda_*,\mu_*}$ 
 with $u=_Gv$ as guaranteed by \cref{prop:reg-hyp}. Construct a new automaton ${\mathcal M}_1$ by adding an edge  labeled ${\#\choose \#}$ from each accept state of $\mathcal M$ to the start state.
 The new automaton ${\mathcal M}_1$ accepts padded pairs of  $(G,S,\lambda_*, \mu_*)$-\qgeod\ words  $(u_i,v_i)$ with  $u_i=_Gv_i$,  and each pair is separated by ${\#\choose \#}$. 
 Make a copy ${\mathcal M}_2$ of ${\mathcal M}_1$ where each letter from $\Sigma_1$ is replaced by its corresponding marked letter from $\Sigma_2$.
  The  automaton ${\mathcal M}_2$ accepts padded pairs of  $(\Gamma(G,S_\dag),\lambda_*, \mu_*)$-\qgeod\ words  $(u_i,v_i)$ with  $u_i=_Gv_i$,  and each pair is separated by ${\#_\dag\choose \#_\dag}$. 
  
  Make a new automaton $\mathcal M_3$ by attaching  edges labeled $\varepsilon$ from the accept states of $\mathcal M_1$ to the start state of $\mathcal M_2$.
  Note that the size of  $\mathcal M_3$ is constant.

 We now take $L_1=L(\mathcal M_3)\cap \psi^{-1}(L_{\textrm{cover}})$. This is again ET0L in \NSPACE$(s(n))$ by \cref{prop:closureET0L}(3). The language $L_1$ can be seen as pairs of strings, the top string a padded version of a string in $L_{\textrm{cover}}$  which has the form 
 \[w_1\#\dots \#w_rf(z_1)\#_\dag\dots \#_\dag f(z_r)\] and the bottom of the form \[u_1\#\dots \#u_r f(v_1)\#_\dag \dots \#_\dag f(v_r)\] with $w_i=_Gu_i=_Gv_i$ and $u_i,v_i$ 
 $(G,S,\lambda_*,\mu_*)$-\qgeods\ for all $1\leq i\leq r$.

 Define a homomorphism $\xi\colon L_1  \to (\Sigma_1\cup \Sigma_2)^*$ by 
 \[ \xi\left({x\choose y}\right)=\left\{\begin{array}{llll} 1 & \ \ & y=\$\\y & & y\in \Sigma_1\cup \Sigma_2\end{array}\right.\]  
Then 
by \cref{prop:closureET0L}(2),  $\xi(L_1)$  is ET0L in \NSPACE$(s(n))$, and consists of all possible strings of the form  \[u_1\#\dots \#u_r f(v_1)\#_\dag \dots \#_\dag f(v_r)=u_1\#\dots u_rf(v_1\#\dots \#v_r)\] with  $u_i,v_i$ 
 $G,S,\lambda,\mu)$-\qgeods, $u_i=_Gv_i$, and such that there exists some $w_1\#\dots \#w_rf(z_1\#\dots\#z_r)\in L$ with $w_i=_Gu_i$.

Finally,   $L_{2}=\xi(L_1)\cap \mathcal T_{\#r}$ is  ET0L in \NSPACE$(s(n))$ since $\mathcal T_{\#r}$ is regular and constant size, and  
since $\langT$ a subset of $Q_{G,S,\lambda',\mu'}\subseteq Q_{G,S,\lambda_*,\mu_*}$ since $\lambda'\leq \lambda_*$ and $\mu'\leq \mu_*$, $L_2$ consists of all words of the form \[u_1\#\dots u_rf(v_1\#\dots \#v_r)\] for all possible $u_i,v_i\in\langT$ where $u_i=_G v_i$ such that there exists some \[w_1\#\dots \#w_r f(z_1\#\dots \# z_r)\in L\] with $w_i=_Gu_i$. Apply the homomorphism $\chi\colon(\Sigma_1\cup\Sigma_2)^*\to \Sigma_1^*$ 
defined by  \[ \chi(x) =\left\{\begin{array}{llll} 1 & \ \ & x\in\Sigma_2\\x & & x\in \Sigma_1\end{array}\right.\]  
to obtain $\chi(L_2)=L_{\mathcal T}$ is ET0L in \NSPACE$(s(n))$.
This establishes the first item.

If  $\mathcal T$  is in bijection with $G$, we have $u_i$ and $v_i$ are identical words, thus $L_{\mathcal T}$ consists of  words of the form $u_1\#\dots \#u_r f(u_1\#\dots \#u_r)$. Thus by  \cref{lem:copyME}, $L_{\mathcal T}$  is EDT0L in 
 \NSPACE$(s(n))$. This establishes the second item.
 
 For the final result, let \[L_G=\left\{(g_1,\dots ,g_r) \ \middle | \  \exists w_1\#\dots \#w_rf(z_1\#\dots \#z_r)\in L [ w_i=_Gg_i]\right\}.\] Then since for each group element there are a finite number of $(G,S,\lambda,\mu)$-\qgeod\ words representing it, $L_{\text{cover}}$ is finite  (resp. empty) if and only if $L_G$ is finite  (resp. empty), and since 
 for each group element there are a finite number of $(G,S,\lambda',\mu')$-\qgeod\ words representing it,
$L_G$ is finite  (resp. empty)  if and only if $L_{\langT}$ is finite (resp. empty).
\end{proof}

\subsection{Alternative generating sets for free and virtually free groups}

Previous work \cite{CDE, DEijac} expresses the full set of (group element) solutions to systems in terms of specific normal forms:  freely reduced words in free groups or the standard normal forms of  \cite[Definition 14.1]{DEijac} for virtually free groups. However, we can express the solutions more generally, in terms of arbitrary languages of \qgeods, as E(D)T0L in the same space complexity as in  \cite{CDE, DEijac}, by using \cref{prop:shortlex-hyp}.

We first observe that the standard normal forms used  in \cite{DEijac} are  \qgeods.
\begin{proposition}[Standard normal forms are \qgeods]\label{std-DE-quasigeod}
Let $V$ be a virtually free group, $F$ a free normal subgroup of $V$, and $H$ a finite quotient that satisfy the exact short sequence
\[1 \rightarrow F \rightarrow V \rightarrow H \rightarrow 1.\]
Consider a symmetric generating set $Y = A \cup (H \setminus \{1_H\})$ for $V$, where $A$ is a free basis generating set for $F$.

The set $\{uh \mid u \in A^* \text{ freely reduced over } A, h \in H \setminus 1_H\}$ is a set of unique $(V,Y,\lambda_Y, \mu_Y)$-\qgeod\ representatives for $V$ for some $\lambda_Y \geq 1, \mu_Y\geq 0$. 
\end{proposition}\begin{proof} 
Since $F$ is normal in $V$, for each $a\in A, h\in H$ we can write $hah^{-1}=u_{a,h}$ where $u_{a,h}\in A^*$ is some fixed choice of word. Let $m=\max\{|u_{a,h}|\mid a\in A, h\in H\}$ and $h_i'=_V h_0h_1\dots h_i$.

Suppose $g\in V$ has standard normal form $uh$, and suppose $v=h_0 a_1 h_1 \cdots a_k h_k$ is a geodesic for $g$, 
where  $a_i\in A$ and $h_i\in H\cup\{1_H\}$. We have $ |v|_Y \geq k$. 
Now push the $h_i$ to the right: \[h_0a_1h_1a_2\dots h_k=_V u_{a_1,h_0}h_0h_1a_2\dots h_k=_V u_{a_1,h_0}h_1'a_2\dots h_k= \] \[u_{a_1,h_0}u_{a_2,h_1'}h_2'\dots h_k= \dots =_V u'h, \] where $|u'h|_Y \leq mk+1 \leq m |v|_Y +1=m\Abs{g}_Y +1$. Since $ |u|_Y \leq |u'|_Y$ we get that $\Abs{g}_Y \leq |uh|_Y \leq m \Abs{g}_Y+1$, so $uh$ is \qgeod\ in $V$  for some constants $\lambda_Y \geq 1, \mu_Y\geq 0$. 
\end{proof}

Then \cref{prop:shortlex-hyp} implies the following, which extends previous results for free and virtually free groups by both allowing arbitrary generating sets $S$ instead of very specific ones, and by expressing the solutions in terms of more general quasigeodesics instead of particular normal forms.
\begin{restatable}[Free and virtually free with arbitrary generating sets]{corollary}{changeGsetCOR}
\label{cor:changeGset}
Let $G$ be a free [respectively virtually free] group with finite symmetric generating set $S$.   
Let $\langT\subseteq Q_{G,S,\lambda, \mu}$ be a regular language of $(G,S, \lambda,\mu)$-\qgeod\ words over $S$ surjecting to $G$ 
 for some fixed arbitrary values of
 $\lambda\geq 1, \mu\geq 0, \lambda\in\mathbb Q$.

Let $\Phi$ be a system of equations and inequations
of size $n$ with constant size rational constraints, as in \cref{defn:system}. 
Then \begin{enumerate}\item 
the full set of $\langT$-solutions
is ET0L, and the algorithm which on input $\Phi$ prints a description for the ET0L grammar runs in  $\NSPACE(n\log n)$ [resp. $\NSPACE(n^2\log n)$];
\item if  $\langT$ is in bijection with $G$, then
the full set of $\langT$-solutions
is EDT0L, and the algorithm which on input $\Phi$ prints a description for the EDT0L grammar runs in  $\NSPACE(n\log n)$ [resp. $\NSPACE(n^2\log n)$].
\end{enumerate}
\end{restatable}

\begin{proof} Let $S_1$ be a free basis generating set for a free group $F$ and $R_1$ the set of all 
freely reduced words over $S_1\cup S_1^{-1}$.
We have $R_1$ are  $(G,S_1\cup S_1^{-1},1,0)$-\qgeods. By \cref{thmCDE} the set of all solutions to $\Phi$ 
as words in $R_1$ is EDT0L in $\NSPACE(n\log n)$. 
Applying the homomorphism which sends each generator to a word in $S^*$, we obtain (by \cref{prop:closureET0L}(2) and \cref{change_of_gset}) a covering solution of words over $S^*$ of $(G,S, \lambda', \mu')$-\qgeods\ for some $\lambda', \mu'$.
Next,
apply \cref{lem:doubleME} to obtain a doubled copy which is EDT0L also in  $\NSPACE(n\log n)$. We can now input this into   \cref{prop:shortlex-hyp} to obtain the result.

 For systems over virtually free groups, let $Y$ be the generating set used in   \cite{DEijac} and $R_1$ the set of standard normal forms over $Y$. By
  \cref{std-DE-quasigeod}  we have that standard normal forms over $Y$ are \qgeods.
   By \cref{thmDE} the set of all solutions to $\Phi$ 
as words in $R_1$ is EDT0L in $\NSPACE(n^2\log n)$. 
  Applying the homomorphism which sends each letter of $Y$ to a word in $S^*$, we obtain (by \cref{prop:closureET0L}(2) and \cref{change_of_gset}) a covering solution of words over $S^*$ of $(G,S, \lambda', \mu')$-\qgeods\ for some $\lambda', \mu'$.
Next,
apply \cref{lem:doubleME} to obtain a doubled copy which is EDT0L also in  $\NSPACE(n^2\log n)$. We can now input this into   \cref{prop:shortlex-hyp} to obtain the result.
 \end{proof}

\section{A key fact about  \Qier\ constraints and complexity}\label{sec:DG9.4issue}

In  this paper we will make use of the following result of \DG, which we restate here together  with a note about complexity.
\begin{proposition}[{see \cite[Proposition 9.4]{DG}}]\label{prop:DG9.4complexity}
Let $G$ be a hyperbolic group with finite symmetric generating set $S$, and let $\mathcal R\subseteq G$ be an effective \qier\ set (recall this means we are given an NFA $A$ such that there exist constants $\lambda_0\geq 1,\mu_0\geq 0$ with  $L(A)\subseteq Q_{G,S,\lambda_0,\mu_0}$ and $\pi(L(A))=\mathcal R$). 
Then for $\lambda\geq 1,\mu\geq 0$ fixed given constants,
 there is an algorithm  which computes an automaton $A'$ that accepts the language \[\widetilde{\mathcal R}=\pi^{-1}(\mathcal R)\cap Q_{G,S,\lambda,\mu}.\] 
 If the  set $\mathcal R$ is just effective \qier\ and not explicit, then 
in general there is no {computable}  bound on the number of steps that  the subroutine must make before it terminates. 
\end{proposition}
\begin{proof}
The main statement is proved in full detail in \cite[Proposition 9.4]{DG}. We start by assuming we have an NFA $A$ such that there exist constants $\lambda_0\geq 1,\mu_0\geq 0$ with  $L(A)\subseteq Q_{G,S,\lambda_0,\mu_0}$ and $\pi(L(A))=\mathcal R.$ 
We don't assume we are given the $\lambda_0,\mu_0$ {\em a priori}, just that they exist.
A key step in their algorithm is to calculate some constants which can play the roles of $\lambda_0,\mu_0$, which is done by enumerating larger and larger values of constants $\lambda_1,\mu_1$ until a certain condition is satisfied. The proof of \cite[Proposition 9.4]{DG} gives no bound on how many steps or how much space this would take; all it needs is that the enumeration is guaranteed to terminate since we are promised that $L(A)\subseteq Q_{G,S,\lambda_0,\mu_0}$ for some  $\lambda_0,\mu_0$.
 The automaton $A'$ is then constructed using
constants $\lambda',\mu'$ which themselves depend on $\lambda_1,\mu_1$. The size of $A'$ is 
a
 function of the constants $\lambda',\mu'$ and $|A|$. 
 If $|A|$ is considered a considered in the context of a larger algorithm, then even though this function may be large (or even not bounded, as we discuss in the next paragraph), the complexity of this subroutine to compute $A'$ is considered to be constant.
 In the case of an explicit \qier\ set, where the $\lambda_0,\mu_0$ are also specified, and also considered constants in the context of a larger algorithm, then it is clear that the function taking inputs $|A|, \lambda_0,\mu_0$ is also constant. 
 In the case that $\lambda_0,\mu_0$ are not given, it may seem dishonest to claim the subroutine makes no contribution, but since this is a standard convention in complexity theory and it leads to stronger statements, we follow this convention.

Suppose hypothetically there was a computable function $s:\mathbb N\to\mathbb R$ such that on  input 
 $(G,\lambda,\mu,A)$, some algorithm was able to compute an automaton $A'$ accepting $\widetilde{\mathcal R}=\pi^{-1}(R)\cap Q_{G,S,\lambda,\mu}$. This leads to a contradiction: 
let $H$ be a finitely generated subgroup of $G$ with generators $u_1,\dots, u_n\in S^*$ and construct its Stallings graph, which is an automaton $\mathcal M$ with $\pi(L(\mathcal M))=H$ (by folding the bouquet of loops labeled by the $u_i$). Now {\em a priori} we do not know whether or not $H$ is  quasi-isometrically embedded. 
 Consider the system $\Phi$ consisting of a single equation $X=X$ and constraint $X\in L(\mathcal M)$. Run the hypothetical algorithm to solve $\Phi$ via constructing the automata accepting all $(G,S,\lambda,\mu)$-\qgeod\ words corresponding to the  constraint, and if it does not produce an answer within space $\Oh(s(n))$ (equivalently in time $2^{\Oh(s(n))}$ steps) then we know $H$ was not quasi-isometrically embedded, and if it does then $H$ is quasi-isometrically embedded and we have even found  the appropriate constants.
However, by \cite{BWmalnormal} the problem of deciding whether or not $H$ is quasi-isometrically embedded for an arbitrary hyperbolic group is undecidable, so there is no computable bound on the number of steps needed to determine whether $H$ is not quasi-isometrically embedded.
\end{proof}

In this paper we consider systems of equations and inequations with \qier\ constraints given effectively by an NFA for each variable, as in \cref{defn:system}. If we consider the sizes of the NFAs to be  constant, with the input size depending only on 
the size of the equations and inequations, then when running the algorithm in the previous proposition as a subroutine inside a larger algorithm, it is standard practice in complexity theory to consider the subroutine as requiring only constant space since its complexity depends solely on the sizes of the input NFAs. That is, even though {\em a priori} we cannot bound the number of steps that may be needed to compute $\lambda_0,\mu_0$, the subroutine only depends on items considered to be constant. 
For this reason, the statements of our main theorems refer to constant size effective \qier\ constraints, rather than just explicit.

\section{Equations and their solutions in torsion-free hyperbolic groups}\label{sec:torsionfree}  

In this section we prove \cref{thmTorsionFree}, which characterises the solutions of equations in torsion-free hyperbolic groups as EDT0L in \PSPACE. The first and main part of the algorithm for getting the solutions is \cref{prop:MAIN-torsionfree}, where by \cref{CylStab} we obtain a covering solution set of \qgeod\ words,  but not yet as shortlex or other representatives. 
\Cref{prop:shortlex-hyp}
then allows us get the solutions as shortlex representatives from the set produced in \cref{prop:MAIN-torsionfree}. 
The algorithm relies on the existence and properties of canonical representatives, explained below, which 
 guarantee that the solutions of any system in a torsion-free hyperbolic group with generating set $S$ can be found by solving an associated system in the free group on $S$. 
 
 \subsection{Canonical representatives} 

Canonical representatives of elements in torsion-free hyperbolic groups were defined by Rips and Sela in \cite{RS95}, and we refer the reader to \cite{RS95} for their construction, and to \cite[Appendix C]{CiobanuEicalp2019}
for a basic exposition. 

Let $G$  be a torsion-free hyperbolic group with generating set $S$, and let $g\in G$. For a fixed integer $T\geq 1$, called the \textit{criterion}, the canonical representative $\theta_T(g)$ of $g$ with respect to $T$ is a $(G,S,\lambda, \mu)$-\qgeod\ word over $S$ which satisfies $\theta_T(g)=_G g$, $\theta_T(g)^{-1}=\theta_T(g^{-1})$. If $T$ is well-chosen, a number of combinatorial stability properties make canonical representatives with respect to $T$ particularly suitable for solving triangular equations in hyperbolic groups, as in \cref{CylStab}.

It is essential for the language characterisations in our main results that canonical representatives with respect to $T$ are $(G,S,\lambda, \mu)$-\qgeods, with $\lambda$ and $\mu$ depending only on $\delta$, and not on $T$.

\begin{proposition}
[{see  \cite[Proposition 3.4]{DahmaniIsrael}}]
\label{canrep_qg} There are constants $\lambda_G \geq 1$ and $\mu_G \geq 0$ 
depending only on $G$ such that for any criterion $T$ and any element $g \in G$ the canonical representative $\theta_T(g)$ of $g$ is a $(G,S,\lambda_G, \mu_G)$-\qgeod.
\end{proposition}

The following result provides the main reduction of a system of equations in a torsion-free hyperbolic group to a system over the free group with the same basis.
 
 \begin{theorem}[{\cite[Theorem 4.2, Corollary 4.4]{RS95}}]\label{CylStab}
 Let $G$ be a torsion-free $\delta$-hyperbolic group generated by set $S$, and let $\Phi=\{\phi_j(\mathcal X, \mathcal{A})=1\}_{j=1}^q$ be a system of triangular equations, that is, $\phi_j(\mathcal X, \mathcal{A})=X_{(j,1)}X_{(j,2)}X_{(j,3)}=1,$
 where $1\leq j \leq q$ and $(j,a)\in \{1, \dots l\}$. 
 
 There exists an effectively computable criterion $T$, a
  constant $b=b(\delta, q)$ depending on $\delta$ and linearly on $q$ such that:
 if $(g_1, \dots, g_l) \in \text{Sol}_G(\Phi)$, then 
 there exist $y_{(j,a)}, c_{(j,a)} \in F(S)$, $1 \leq j \leq q, 1\leq a\leq 3$, such that $$|c_{(j,a)}|_S \leq b \textrm{ and } c_{(j,1)}c_{(j,2)}c_{(j,3)}=_G 1,$$ for which the canonical representatives satisfy:
  $$\theta_T(g_{(j,1)})=y_{(j,1)}c_{(j,1)} (y_{(j,2)})^{-1},$$ 
   $$\theta_T(g_{(j,2)})=y_{(j,2)}c_{(j,2)}(y_{(j,3)})^{-1},$$
    $$\theta_T(g_{(j,3)})=y_{(j,3)}c_{(j,3)} (y_{(j,1)})^{-1}.$$
   \end{theorem}

The following explains how \cref{CylStab} is employed.
\subsection{Reduction from torsion-free hyperbolic to free groups: an overview.} 
The first part of the algorithm that finds solutions to a system of equations and inequations is to reduce it to a triangular one, that is, where each equation has length $3$. In the free group $F(S)$, a triangular equation such as $XY=Z$ 
 has a solution in reduced words over $S$ if and only if there exist words $P,Q,R$  with $X=PQ, Y=Q^{-1}R, Z=PR$ where no cancellation occurs between $P$ and $Q$, $Q^{-1}$ and $R$, and $P$ and $R$. 
 (This allows the translation of equations in free groups into equations in free monoids with involution, as in \cite[Lemma 4.1]{CDE}). 
 
  In a hyperbolic group, the direct reduction of a triangular equation to a system of cancellation-free equations is no longer possible. Instead of looking for geodesic solutions $(g_1,g_2,g_3)$ in $G$ to the equation $XY=Z$, represented by a geodesic triangle in the Cayley graph of $G$  as in \cref{subfig:geod}, one looks for solutions as in \cref{CylStab}, illustrated in \cref{subfig:can}. That is, one finds the solution $g_i$ via some canonical representative $\theta_T(g_i)$, where $T$ is computable and guaranteed to exist; we henceforth write $\theta(-)$ for $\theta_T(-)$. By \cref{CylStab} there exist $y_i, c_i \in F(S)$ such that the equations $\theta(g_1)=y_1c_1y_2$, $\theta(g_2)=y_2^{-1}c_2y_3$, $\theta(g_3)=y_3^{-1}c_3y_1^{-1}$ are cancellation-free (no cancellation occurs between $y_1$ and $c_1$, $c_1$ and $y_2$ etc.) equations over $F(S)$. Moreover,  \cref{CylStab} implies that the $y_i$'s should be viewed as `long' prefixes and suffixes that coincide, and the $c_1 c_2 c_3$ as a `small' inner circle with circumference in $\Oh(n)$ (see \cref{subfig:can}).

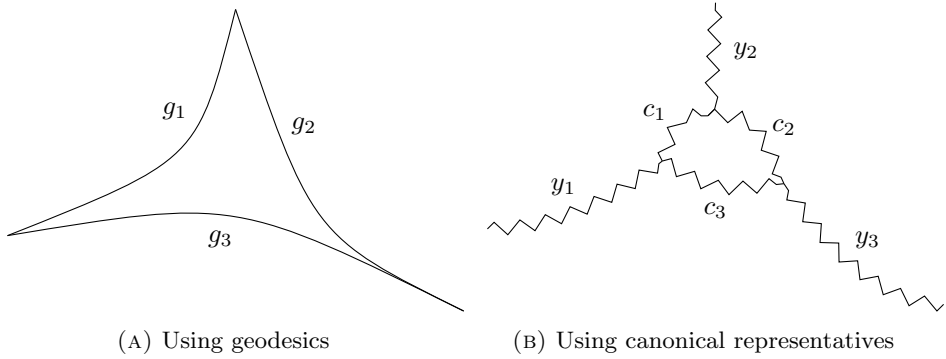
\begin{figure}[ht]
    \centering
     \begin{subfigure}[t]{0.45\textwidth}
        \begin{tikzpicture}[font=\sffamily]
\draw[black] (0,0) .. controls  (2.5,1) ..  (3,3) ;
\draw[black]  (0,0) .. controls  (3,.5) .. (6,-1);
\draw[black]  (6,-1) .. controls  (4,0) .. (3,3);
  
\draw (2.2,1.3) node [label=$g_1$]  {};
\draw   (3.9,1.1) node [label=$g_2$]{};
\draw   (2.8,-.4) node [label=$g_3$] {}; 
  
\end{tikzpicture}
        \subcaption{Using geodesics}\label{subfig:geod}
    \end{subfigure}\hspace{5mm}
       \begin{subfigure}[t]{0.45\textwidth}
               \begin{tikzpicture}[font=\sffamily]
\draw (0,0)[decorate, decoration=zigzag]  parabola (2.3,.9);
\draw (6,-1)[decorate, decoration=zigzag]    parabola (3.9,.6);
\draw (2.9,1.5)[decorate, decoration=zigzag]  parabola  (3,3);

\draw[decorate, decoration=zigzag] (2.3,.9)  .. controls (3,.5) .. (3.9,.6);
\draw[decorate, decoration=zigzag] (2.9,1.5)   .. controls (3.5,1.5) ..(3.9,.6);
\draw[decorate, decoration=zigzag] (2.3,.9)   .. controls (2.5,1.5) ..(2.9,1.5);

\draw (2.2,1.2) node [label=$c_1$]  {};
\draw   (3.9,1) node [label=$c_2$]{};
\draw   (3,-.1) node [label=$c_3$] {};

\draw (1,.2) node [label=$y_1$]  {};
\draw   (3.4,2) node [label=$y_2$]{};
\draw   (5,-.5) node [label=$y_3$] {};

\end{tikzpicture}

      \subcaption{Using canonical representatives}\label{subfig:can}
    \end{subfigure}
    \caption{Solutions to $XY=Z$ in a hyperbolic group.}\label{fig:triangle}
\end{figure}

We now prove \cref{prop:MAIN-torsionfree}, which together with \cref{prop:shortlex-hyp} gives \cref{thmTorsionFree}. This proposition provides a covering set of solutions (that is, every solution as tuple of group elements is found, but not necessarily written as shortlex representatives etc) that can then be processed via \cref{prop:shortlex-hyp} to give the EDT0L formal language description in the space complexity claimed in \cref{thmTorsionFree}.

\begin{proposition}\label{prop:MAIN-torsionfree}
Let $G$ be a torsion-free hyperbolic group with finite symmetric generating set $S$.  Let $\#$ be a symbol not in $S$, and let $\Sigma=S\cup\{\#\}$.  
Let $\Phi$ be a system of equations and inequations
of size $n$ (without rational constraints) 
over variables $X_1,\dots,  X_r$, as in \cref{defn:system}.

\begin{enumerate}
\item There exists a language
$L=\{w_1\#\dots\# w_r \}$ over $\Sigma$
such that \begin{enumerate}\item $L$ is a covering solution set, 
\item $ w_i \in Q_{G,S,\lambda,\mu}, 1\leq i \leq r $ where $\lambda=\lambda_G, \mu=\mu_G$ are the (fixed) constants from \cref{canrep_qg},
 \item  $L$ is  EDT0L  in $\NSPACE(n^2\log n)$.
 \end{enumerate}
\item The system $\Phi$ has infinitely many solutions if and only if $L$ is infinite, and no solution if and only if $L$ is empty.
\end{enumerate}
\end{proposition}
\begin{proof} 
We produce a language $L$ of \qgeod\ words over $S$ as above such that any tuple $(\pi(w_1),\dots,\pi(w_r))$ arising from some $w_1\#\dots\# w_r$ in $L$ is in the solution set $\text{Sol}_G(\Phi)$ (soundness). We then prove 
(using \cref{CylStab}) that any solution in $\text{Sol}_G(\Phi)$ is the projection, via $\pi$, of some element in $L$ (completeness).

\medskip

\noindent \textbf{Steps 1 and 2: Preprocessing}
 \begin{itemize}
 \item[Step 1][Remove inequations]
  We first transform $\Phi$ into a system consisting entirely of equations by adding a variable $X_N$ to $\mathcal{X}$ and replacing any inequation $\varphi_j(\mathcal X, \mathcal{A})\neq1$ by $\varphi_j(\mathcal X, \mathcal{A}) = X_N$, with the constraint $X_N\neq_G 1$.

\item[Step 2] [Triangulate] 
We transform each equation into several equations of length $3$, by introducing new variables. This can always be done (see the discussion in \cite[\S{4}]{CDE}), and it produces approximately $\sum_{i=1}^s l_i\in \Oh(n)$ triangular equations with set of variables $\mathcal{Z}$
 where $|\mathcal Z|\in \Oh(n)$ and $\mathcal X \subset \mathcal{Z}$. 
From now on assume that the system $\Phi$ consists of $q\in \Oh(n)$ equations of the form
 $X_jY_j=Z_j$ where $1\leq j \leq q$. 
\end{itemize}
\noindent\textbf{Steps 3 and 4: Lifting $\Phi$ to the free group on $S$}

 \begin{itemize}
 \item[Step 3] Let $b \in \Oh(q)= \Oh(n)$ be the constant from \cref{CylStab} that depends on $\delta$ and linearly on $q$. We run in lexicographic order through all possible tuples of words $\mathbf c=(c_{11},c_{12},c_{13},\dots, c_{q1},c_{q2},c_{q3})$ with $c_{ji}\in S^*$ and $|c_{ji}|_S\leq b$.
 \item[Step 4]
 For each tuple $\mathbf c$ we use Dehn's algorithm to check whether $c_{j1}c_{j2}c_{j3}=_G1$; if this holds for all $ j $ we construct a system $\Phi_{\mathbf c}$ of equations of the form 
\begin{equation}\label{csystem}
X_j=P_jc_{j1}Q_j, \ Y_j=Q_j^{-1}c_{j2}R_j, \ Z_j=P_jc_{j3}R_j, \ \ 1\leq j\leq q,
\end{equation}
where $P_i, Q_i, Z_i$ are new variables. This new system has size $\Oh(n^2)$. 
 In order to avoid an exponential size complexity we write down each system $\Phi_{\mathbf c}$ one at a time, so the space required for this step is $\Oh(n^2)$.  
\end{itemize}
Let $\mathcal{Y}\supset\mathcal{Z}\supset \mathcal X$, $|\mathcal{Y}|=m$, be the new set of variables, including all the $P_i, Q_i, Z_i$. 
 \begin{itemize}
 \item[Soundness]Any solution to  $\Phi_{\mathbf c}$ 
in the free group $F(S)$ is clearly a solution to $\Phi$ in the original hyperbolic group $G$ when restricting to the original variables $\mathcal X$. That is, if $(w_1, \dots, w_m) \subseteq F^m(S)$ is a solution to $\Phi_{\mathbf c}$, then $(\pi(w_1), \dots, \pi(w_r))$ is a solution to $\Phi$ in $G$. This shows soundness.
\end{itemize}

Now let $\lambda_G \geq 1, \mu_G\geq 0$ be the constants provided by \cref{canrep_qg}.
Note that if a word $w\in S^*$ is a $(G,S,\lambda_G,\mu_G)$-\qgeod\ and satisfies $w=_G1$ then $\abs{w}_S \leq \mu_G$. We can construct an NFA $\mathcal N$ which accepts all words in $S^*$ equal to $1$ in the hyperbolic group $G$ of length at most $\mu_G$  in constant space (using for example Dehn's algorithm), and in our next step we will use this rational constraint to handle
the variable $X_N$ added in Step 1 (to remove inequalities).

\medskip

\noindent \textbf{Steps 5 and 6: Solving equations with constraints in $F(S)$}

\medskip

\begin{itemize}
 \item[Step 5] 
We now run  the algorithm from \cite{CDE} (which we will refer to as the CDE algorithm) which takes input $\Phi_\mathbf{c}$ of size $\Oh(n^2)$, plus the rational constraint $X_N\not\in  L(\mathcal N)$, plus for each $Y\in \mathcal Y$ the rational constraint that the solution for  $Y$ is a word in $Q_{G,S,\lambda_G, \mu_G}$.  
Since these constraints have constant size (depending only on the group $G$, not the system $\Phi$),  they do not contribute to the $\Oh(n^2)$ size of the input to the CDE algorithm.

We modify the CDE algorithm to ensure every node printed by the algorithm includes the additional label $\mathbf c$. (This ensures the NFA we print for each system $\Phi_{\mathbf c}$ is distinct.)
This does not affect the complexity since $\mathbf c$ has size in $\Oh(n^2)$. Let $\mathcal C$ be the set of all tuples ${\mathbf c}$.
 \item[Step 6] We run the modified CDE algorithm described above  (adding label $\mathbf c$ for each node printed) to print an NFA (possibly empty) for each $\Phi_{\mathbf c}$, which is the rational control for an EDT0L grammar that produces all solutions as freely reduced words in $F(S)$, which correspond to solutions as $(G,S,\lambda_G,\mu_G)$-\qgeods\ to the same system $\Phi_{\mathbf c}$ in the hyperbolic group.

\end{itemize}

\begin{itemize}
\setcounter{enumi}{7}

\item[Completeness] If $(g_1,\dots, g_r)\in  G^r$ is a solution to $\Phi$ in the original hyperbolic group,  \cref{CylStab} 
guarantees that there exist canonical representatives $w_i\in Q_{G,S,\lambda_G, \mu_G}$ with $w_i=_Gg_i$, which have freely reduced forms $u_i=_Gw_i$ for $1\leq i\leq m$, and our construction is guaranteed to capture any such collection of words. More specifically, if $(w_1,\dots, w_r)$ is a solution in canonical representatives to $\Phi$ 
then $(u_1,\dots, u_r,\dots u_{|\mathcal Y|})$  will be included in the solution to $ \Phi_{\mathbf c}$ produced by the CDE algorithm, with $u_i$ the reduced forms of $w_i$ for $1\leq i\leq m$. This shows completeness once we union the grammars from all systems $\Phi_{\mathbf c}$ together.

\item[Step 7] [Producing the solutions in $G$] Add a new start node with edges to each of the start nodes of the NFA's with label $\mathbf c$, for all $\mathbf c \in \mathcal C$. We obtain an automaton which is the rational control for the language $L$ in the statement of the proposition. This rational control makes $L$ an EDT0L language of $(G, S, \lambda, \mu)$-\qgeods \ as required, where $\lambda:=\lambda_G, \mu:=\mu_G$. 
The space needed is exactly that of the CDE algorithm on input $\Oh(n^2)$, which is $\NSPACE(n^2\log n)$.
\end{itemize}

To see why (2) in the statement of the proposition holds, note that when we call on the CDE algorithm above, this will be able to report if the system over $F(S)$ has infinitely many, finitely many, or no solutions. If CDE reports finitely many or no solutions then we are done, as $L$ will be finite or empty, respectively. If CDE returns infinitely many solutions, note that there are only finitely many words over $S$ per group element in $G$ in the solution set because (in Step 5) we run CDE with the rational constraint that only $(G,S,\lambda_G,\mu_G)$-\qgeods\ are part of output, so we end up with infinitely many solutions in $\text{Sol}_G(\Phi)$. 
 \end{proof}

We can now prove our main result about torsion-free hyperbolic groups. 

\begin{theorem}[Torsion-free]\label{thmTorsionFree}
Let $G$ be a torsion-free hyperbolic group with finite symmetric generating set $S$.  Let $\#$ be a symbol not in $S$, and let $\Sigma=S\cup\{\#\}$.  
Let $\Phi$ be a system of equations and inequations
of size $n$ with constant size effective  \qier\ constraints over variables $X_1,\dots,  X_r$, as in \cref{sec:notationSolns,defn:explicit}.

Then we have  the following.

 \begin{enumerate}
     \item\label{item:THMsurjectSet}  For any $\lambda'\geq 1,\mu'\geq 0, \lambda'\in\mathbb Q$ and for any regular language $\langT \subseteq Q_{G,S,\lambda', \mu'}$ such that the projection $\pi:\langT\to G$ is a surjection,
the   full set of $\langT$-solutions
 \[\text{Sol}_{\langT,G}(\Phi)=\{(w_1\#\dots\# w_r ) \mid w_i \in \langT,  (\pi(w_1),\dots, \pi(w_r) ) \text{ solves  } \Phi\}\]
   is ET0L  in $\NSPACE(n^2\log n)$.
  \item \label{item:THMshortlex}
  For any $\lambda'\geq 1,\mu'\geq 0, \lambda'\in\mathbb Q$ and for any regular language $\langT\subseteq Q_{G,S,\lambda', \mu'}$ such that the projection $\pi: \langT\to G$ is a bijection,
the   full set of $\langT$-solutions
 \[\text{Sol}_{\langT,G}(\Phi)=\{(w_1\#\dots\# w_r ) \mid w_i \in \langT,  (\pi(w_1),\dots, \pi(w_r) ) \text{ solves  } \Phi\}\]
   is EDT0L  in $\NSPACE(n^2\log n)$.
\item It can be decided in $\NSPACE(n^2\log n)$ whether or not  $\text{Sol}_G(\Phi)$
is empty, finite or infinite. 
\end{enumerate}
\end{theorem}
In particular, if $\langT$ is the set of all shortlex geodesics over $S$, then \Cref{thm:IntroTorsionFree} follows immediately.

\begin{proof} Suppose $R_{X_i}$, $X_i\in \mathcal{X}$, are the  constant size effective \qier\ constraints that are part of the system $\Phi$.
 By  \cref{prop:DG9.4complexity}
 (\cite[Proposition 9.4]{DG})
 the languages $\widetilde{R_{X_i}}=\pi^{-1}(R_{X_i})\cap Q_{G,S,\lambda_G,\mu_G}$ are regular and one can compute automata accepting them in constant space. (Recall, the complexity for this step  only depends on the constant size $R_X$, so from the point of view of our result the complexity is constant.)

Let $\Phi'$ be the system consisting only of the equations and inequations from $\Phi$. 
Apply  \cref{prop:MAIN-torsionfree} to $\Phi'$ to obtain  an EDT0L language  $L$ of covering solutions to $\Phi'$.
Then intersect the language $L$ of covering solutions with $\widetilde{R_{X_1}}\# \dots\#\widetilde{R_{X_r}}$ to ensure that the solutions belong to the input constraints, and call this covering solution set $L_c$ (to point out that constraints have been taken into account). The language $L_c$ is EDT0L, as the intersection of an EDT0L language with a regular language, and stays in the same space complexity of $\NSPACE(n^2\log n)$ by \cref{prop:closureET0L}.
 
 We then apply \cref{lem:doubleME} to the covering set $L_c$ to double it and get a set which satisfies the conditions and plays the role of the set $L_{\text{cover}}$ in \cref{prop:shortlex-hyp}. Then 
\cref{prop:shortlex-hyp} applied to this set shows (1) that the set of solutions written as words in a regular language (of \qgeods) in bijection with the group is EDT0L, and (2) written as words in a regular language (of \qgeods) surjecting onto the group is ET0L. Part (3) also follows immediately from \cref{prop:shortlex-hyp}.
\end{proof}

\section{Equations and their solutions in hyperbolic groups with torsion}\label{sec:torsion}

In the case of a hyperbolic group $G$  with torsion, the general approach of Rips and Sela can still be applied, but the existence of canonical representatives is not always guaranteed (see Delzant \cite[Rem.III.1]{Delzant}). To get around this, \DG\  `fatten' the Cayley graph $\Gamma(G,S)$ of $G$ to a larger graph $\mathcal K$ which contains $\Gamma(G,S)$ (in fact $\Gamma(G,S)$ with midpoints of edges included), and solve equations in $G$ by considering equalities of paths in  $\mathcal{K}$. 
 More precisely, $\mathcal{K}$ is the $1$-skeleton of the barycentric subdivision of a Rips complex of $G$, as explained below.
 
 Let $P_{50\delta}(G)$ be the Rips complex whose set of vertices is $G$, and whose simplices are subsets of $G$ of diameter at most $50\delta$ (see \cref{RipsComplex}). Then let $\mathcal{B}$ be the barycentric subdivision of $P_{50\delta}(G)$ and let $\mathcal{K}=\mathcal{B}^1$ the $1$-skeleton of $\mathcal{B}$. By construction, the vertices of $\mathcal{K}$ (and $\mathcal{B}$) are in $1$-to-$1$ correspondence with the simplices of $P_{50\delta}(G)$, so we can identify $G$, viewed as the set of vertices of $\Gamma(G,S)$, with a subset of vertices of $\mathcal{K}$.  

\begin{remark}\label{rmkKG} The graphs $\mathcal{K}$ and $\Gamma(G,S)$ are quasi-isometric, and any vertex in $\mathcal{K}$ that is not in $G$ is at distance $1$ (in $\mathcal{K}$) from a vertex in $G$.
\end{remark}

\begin{definition} \label{Kpaths}
Let $\gamma, \gamma'$ be paths in $\mathcal{K}$.
\begin{itemize}
\item[(i)] We denote by $i(\gamma)$ the initial vertex of $\gamma$, by $f(\gamma)$ the final vertex of $\gamma$, and by $\overline{\gamma}$ the reverse of $\gamma$ starting at $f(\gamma)$ and ending at $i(\gamma)$.  
\item[(ii)] We say that $\gamma$ is \emph{reduced} if it contains no backtracking, that is, no subpath of length $2$ of the form $e \overline{e}$, where $e$ is an edge. 
\item[(iii)] We write $\gamma \gamma'$ for the concatenation of $\gamma$, $\gamma'$ if $i(\gamma')=f(\gamma)$. 
\item[(iv)] Two paths in $\mathcal{K}$ are \textit{homotopic} if one can obtain a path from the other by adding or deleting backtracking subpaths. Each {\em homotopy class} $[\gamma]$ has a unique reduced representative. 
\end{itemize}
\end{definition}

Let $V$ be the set of all homotopy classes $[\gamma]$ of paths $\gamma$ in $\mathcal{K}$ starting at $1_G$ and ending at a point in $\Gamma(G,S)$ which corresponds to some element of $G$, that is, 
\begin{equation}\label{defV}
V:=\{[\gamma] \mid \gamma \in \mathcal{K}, i(\gamma)=1_G, f(\gamma)\in G \}.
\end{equation} For $[\gamma]$, $[\gamma'] \in V$ define their product $[\gamma][\gamma']:=[\gamma {}_v\gamma']$, where $\gamma {}_v\gamma'$ denotes the concatenation of $\gamma$ and the translate ${}_v\gamma'$ of $\gamma'$ by $v=f(\gamma)$; we will abuse notation and write $\gamma \gamma'$ for the concatenation of $\gamma$ and the translate of $\gamma'$ by $f(\gamma)$, without additional mention of the translation.
Let $[\gamma]^{-1}$ be the homotopy class of ${}_{v^{-1}}\overline{\gamma}$. 

The set $V$ is then a group that projects onto $G$ via the final vertex map $f$, that is,
 $f:V \twoheadrightarrow G$ is a surjective homomorphism. Since $G$ has an action on $\mathcal{K}$ induced by the natural action on its Rips complex, $V$ will act on $\mathcal{K}$ as well. This gives rise to an action of $V$ onto the universal cover $T$ (which is a tree) of $\mathcal{K}$, and \cite[Lemma 9.9]{DG} shows that the quotient $T/V$ is a finite graph (isomorphic to $\mathcal{K}/G$) of finite groups, and so $V$ is virtually free.
 
 We assume that the algorithmic construction (see \cite[Lemma 9.9]{DG}) of a presentation for $V$ is part of the preprocessing of the algorithm, will be treated as a constant, and will not be included in the complexity discussion.

The first step in solving a system $\Phi$ of (triangular) equations in $G$ is to translate $\Phi$ into identities between ($\lambda_1, \mu_1$)-\qgeod\ paths in $\mathcal{K}$ ($\lambda_1$ and $\mu_1$ are defined after \cref{prop:K}) that have start and end points in $G$. These
 paths can be seen as the analogues of the canonical representatives from the torsion-free case. This can be done by Proposition 9.8 \cite{DG} (see \cref{prop:K}). 
The second step in solving $\Phi$ is to express the equalities of \qgeod\ paths in $\mathcal{K}$ in terms of equations in the virtually free group $V$ based on $\mathcal{K}$. Finally, Proposition 9.10 \cite{DG} (see \cref{prop:V}) 
shows it is sufficient to solve the systems of equations in $V$ in order to obtain the solutions of the system $\Phi$ in $G$. 

\begin{notation} \label{not:constants} Let $\lambda_0=400 \delta m_0$, where $m_0$ is the size of a ball of radius $50\delta$ in $\Gamma(G,S)$, $\mu_0=8$ and $b=b(\delta, |\Phi|)$ is a linear function of the size of the system $\Phi$ ($b$ is a multiple of the constant `$bp$' defined in \cite[Theorem 4.2]{RS95}). 
\end{notation}

\begin{proposition}[{\cite[Proposition 9.8]{DG}}]\label{prop:K}
Let $G$ be a $\delta$-hyperbolic group generated by $S$, $\mathcal{K}$ the graph defined above, and $\Phi=\{\phi_j(\mathcal X, \mathcal{A})=1\}_{j=1}^q$ a system of triangular equations on variables $\{X_1, \dots, X_l\}$. That is, $\phi_j(\mathcal X, \mathcal{A})=X_{(j,1)}X_{(j,2)}X_{(j,3)}=1,$
 where $1\leq j \leq q$. 
Let $\lambda_0$, $\mu_0=8$ and $b$ be as in  \cref{not:constants}.

 If $(g_1, \dots, g_l) \in \text{Sol}_G(\Phi)$, then there exist paths $y_{(j,a)}, c_{(j,a)} \in \mathcal{K}$, $1 \leq j \leq q, 1\leq a\leq 3$ and $(j,a)\in \{1, \dots l\}$, such that $$\ell_{\mathcal{K}} (c_{(j,a)})\leq b \textrm{ and } f(c_{(j,1)}c_{(j,2)}c_{(j,3)})=1 \in G,$$ and there exist $(\mathcal{K}, \lambda_0,\mu_0)$-\qgeod\ paths $\gamma(g_{(j,a)})$ joining $1_G$ to $g_{(j,a)}$ in $\mathcal{K}$ such that $\gamma(g_{(j,a)}^{-1})={}_{g_{(j,a)}^{-1}}\overline{\gamma(g_{(j,a)})}$ and
  $$\gamma(g_{(j,1)})=y_{(j,1)}c_{(j,1)}(y_{(j,2)})^{-1},$$ 
   $$\gamma(g_{(j,2)})=y_{(j,2)}c_{(j,2)} (y_{(j,3)})^{-1},$$
    $$\gamma(g_{(j,3)})=y_{(j,3)}c_{(j,3)} (y_{(j,1)})^{-1}.$$
 \end{proposition}
 
  For $\lambda_0, \mu_0$ as in  \cref{prop:K}, let $\lambda_1, \mu_1$ be such that any path $\alpha \gamma \alpha'$ in $\mathcal{K}$, where $\ell_{\mathcal{K}}(\alpha)=\ell_\mathcal{K}(\alpha')=1$ and $\gamma$ is a $(\mathcal{K}, \lambda_0, \mu_0)$-guasigeodesic, is a $(\mathcal{K}, \lambda_1, \mu_1)$-\qgeod. Define
\begin{equation}\label{qgK}
Q_{\mathcal{K}, V, \lambda_1, \mu_1}\footnote{Notation similar to that in Definition \ref{defn:qg}, but the type of paths is different: these are \qgeods\ in the graph $\mathcal{K}$ that represent elements in $V$, not paths in a Cayley graph of $V$.}=\{\gamma \in \mathcal{K} \mid \gamma \textrm{\ is in\ } V \textrm{\ and is a reduced\ } (\mathcal{K}, \lambda_1, \mu_1)-\textrm{\qgeod}\},
\end{equation}
that is, $Q_{\mathcal{K}, V, \lambda_1, \mu_1}$ is the set of those paths in $Q_{\mathcal{K}, \lambda_1, \mu_1}$ with start and endpoints in $G$.
Then for any $L>0$ let $$V_{\leq L}=\{[\gamma] \in V\mid \gamma \textrm{\ reduced\ and\ } \ell_{\mathcal{K}}(\gamma)\leq L \}.$$

The next proposition is a refinement of Proposition \ref{prop:K} in the sense that it makes the same statements about equalities of paths, but this time all the paths are in $V$; this setting allows us to work with equations in the virtually free group $V$.

 \begin{proposition}[{\cite[Proposition 9.10]{DG}}]\label{prop:V}
Let $G$ be a $\delta$-hyperbolic group, $\mathcal{K}$, $\lambda_1, \mu_1$, $b = b(\delta, n)$ and $\Phi=\{\phi_j(\mathcal X, \mathcal{A})=1\}_{j=1}^q$ a system on $l$ variables as above. That is, $\Phi$ consists of equations $\phi_j(\mathcal X, \mathcal{A})=X_{(j,1)}X_{(j,2)}X_{(j,3)}=1,$ where $1\leq j \leq q$ and $(j,a)\in \{1, \dots l\}$. 

\begin{enumerate}
\item[(i)] If $(g_1, \dots, g_l) \in \text{Sol}_G(\Phi)$, then 
there exist $v_{(j, a)}, y_{(j,a)}  \in Q_{\mathcal{K}, V, \lambda_1, \mu_1}$ and $c_{(j,a)} \in V_{\leq b}$ such that $f(c_ {(j,1)}c_{(j,2)}c_{(j,3)})=1$ in $G$, $f(v_{(j, a)})=g_{(j, a)}$ and $v_{(j,a)}=y_{(j,a)} c_{(j,a)} (y_{(j,a+1)})^{-1}$, where $a+1$ is computed modulo $3$.

\item[(ii)] Conversely, suppose a set of elements $\{v_1, \dots, v_l\}\subset V$ is given and for every $1 \leq j \leq q, 1 \leq a \leq 3$ there are $y_{(j,a)} \in V, c_{(j,a)} \in V_{\leq \kappa}$, where $(j,a)\in \{1, \dots l\}$, satisfying $v_{(j,a)}=y_{(j,a)} c_{(j,a)} (y_{(j,a+1)})^{-1}$ and $f(c_{(j,1)} c_{(j,2)} c_{(j,3)})=1$, then the $g_{(j,a)}=f(v_{(j,a)})$ give a solution of $\Phi$ in $G$.

\end{enumerate}
 \end{proposition}

In the virtually free group $V$ we will use the results of Diekert and the second author from \cite{DEijac}. 
Let $Y$ be the generating set of $V$, and let $\langT\subseteq Y^*$ be the set of ``standard normal forms'' for $V$ over $Y$ as in \cref{std-DE-quasigeod}. By \cref{std-DE-quasigeod} the set $\langT$ consists of ($\lambda_Y, \mu_Y$)-\qgeods\ over $Y$, where $\lambda_Y, \mu_Y$ are constants depending on $Y$. Let $$\text{Sol}_{\langT,V}(\Psi)=\{w_1\#\dots\# w_r \mid w_i \in \langT, (\pi(w_1),\dots,\pi(w_r) )\text{ solves } \Psi \text{ in } V\}$$ be the set of $\langT$-solutions in $V$ of a system $\Psi$ of size $|\Psi|=\Oh(k)$; by \cite{DEijac}, $\text{Sol}_{\langT,V}(\Psi)$ consists of ($\lambda_Y, \mu_Y$)-\qgeods\ over $Y$ and is EDT0L in $\NSPACE(k^2\log k)$.

In the following we will need to translate between elements in $V$ over the generating set $Y$, and elements in $G$ over $S$, passing through the graph $\mathcal{K}$, in a way that preserves the EDT0L characterisation of languages. We introduce \cref{notationZ} to facilitate this translation.
\begin{notation} \label{notationZ}
Let $Z$ be some generating set of $V$ and let $\pi\colon Z^* \to V$ be the standard projection map from words to group elements in $V$. 
\begin{enumerate}
\item[(i)][From $(V,Z)$ to $\mathcal{K}$] For each generator $z_i \in Z$ denote by $p_i$ the unique reduced path in $\mathcal{K}$ corresponding to $z_i$, that is, with $i(p_i)=1_G$ and $f(p_i) \in G$ such that $z_i=_V [p_i]$. For any word $w=z_{i_1} \dots z_{i_k}$ over $Z$, denote by $p_w$ the path obtained by concatenating the paths $p_{i_j}$, that is, 
\begin{equation}\label{unreduced}
p_w=p_{i_1} \dots p_{i_k},
\end{equation}
 where $i(p_w)=1_G$, $f(p_w) \in G$, and $w=_V [p_w]$. 
 
 \item[(ii)] For $w \in Z^*$ and $g\in V$ such that $\pi(w)=g$, denoted by $p_g=p_{\pi(w)}$ the unique reduced path in $\mathcal{K}$ homotopic to $p_w$ (which might be unreduced).

\item[(iii)] [From $(V,Z)$ to $(G,S)$] For each $z_i\in Z$, choose a geodesic path/word $\gamma_i$ in $\Gamma(G,S)$ such that $i(\gamma_i)=1_G$ and $f(\gamma_i)=f(p_i) \in G$, where $p_i$ is the reduced path representing $z_i$ in $\mathcal{K}$ (see (i)), and let $\sigma: Z \mapsto S^*$ be the map given by $\sigma(z_i)=\gamma_i$. This can be extended to the substitution $\sigma: Z^* \mapsto S^*$ that associates to each word $w=z_{i_1} \dots z_{i_k}$ over $Z$ a path/word $\gamma_w$ in $\Gamma(G,S)$ obtained by concatenation, that is, 
\begin{equation}\label{Gpath}
\sigma(w)=\gamma_w=\gamma_{i_1} \dots \gamma_{i_k},
\end{equation}
 where $i(\gamma_w)=1_G$ and $f(\gamma_w)=f(p_w) \in G$. 

\end{enumerate}
\end{notation}

\begin{remark}\label{rmk:torsionHarder}
Finding all the group element solutions of a system $\Phi$ over $G$, written as some set of words over $S$ that are not necessarily \qgeod \ or in a normal form, is possible by (a) the work of Dahmani and Guirardel \cite{DG}, who reduce this problem to solving equations in the virtually free group $V$ defined above, and by (b) the algorithm of Diekert and the second author for $V$ \cite{DEijac}. However, obtaining the solutions as an EDT0L language of (\qgeod) normal forms in $G$ requires a lot more effort, and is considerably more intricate than in the torsion-free case. 

In the torsion-free case it is sufficient to solve equations in the free group $F(S)$ on the same generating set as $G$ and it is guaranteed (by \cite{RS95}) that the solutions from the free group contain a \qgeod \ for each group element solution. In the torsion case it is necessary to solve equations in $V$, whose generating set $Y$ is completely different from that of $G$, and furthermore, $V$ contains $G$ with infinite index. Simply taking the solutions over $Y$ given by the DE algorithm does not guarantee to produce the \qgeods \ in the graph $\mathcal{K}$ specified in \cref{prop:V} (1), and therefore does not guarantee \qgeod\ solutions in $G$ over $S$. We therefore need to produce a larger set of \qgeods\ in $V$ over $Y$ representing the solution set of \cite{DEijac}, in order to be certain that there is at least a \qgeod\ with parameters from \cref{prop:V} (1) per solution. This blow-up of the solution set is ET0L, and we can ultimately get EDT0L for the solution set in $G$ only by employing the doubling and copying tricks of \cref{prop:shortlex-hyp}, which makes  \cref{prop:MAIN-torsion} more technical than its torsion-free counterpart.

\end{remark}

\begin{proposition}\label{prop:MAIN-torsion} 
Let $G$ be a hyperbolic group with finite symmetric generating set $S$.  
For $\#$ a symbol not in $S$, extend the bijection $f\colon S\to S_\dag$ defined in \cref{notation:doubleGset} to include  $f(\#)=\#$.
Let $\Phi$ be a system of equations and inequations  (without rational constraints) 
of size $n$ over variables $X_1,\dots,  X_r$ (see \cref{defn:system}).  
 Then there exist $\lambda_G\geq 1, \mu_G\geq 0 $ which depend only on $G$, and a language 
\[L\subseteq \{w_1\#\dots\# w_rf(z_1\#\dots \#z_r)\mid w_i ,z_i\in Q_{G,S,\lambda,\mu}, \pi(w_i)=\pi(z_i), 1\leq i \leq r \}\]
such that \begin{enumerate}\item $\{w_1\#\dots\# w_r\mid w_1\#\dots\# w_rf(z_1\#\dots \#z_r)\in L\}$ is a covering solution set for $\Phi$,
 \item  $L$ is  ET0L  in $\NSPACE(n^4\log n)$.
 \end{enumerate}
Moreover, the system $\Phi$ has infinitely many solutions if and only if $L$ is infinite, and no solutions if and only if $L$ is empty.
\end{proposition}

\begin{proof}
The algorithm to produce the solutions for $\Phi$ is similar to the torsion-free case, except that the virtually free group $V$ and the graph $\mathcal{K}$ will play the role of the free group on $S$ and the Cayley graph of $G$ from the torsion-free case, instead. 

\medskip

\noindent \textbf{Steps 1 and 2: Preprocessing} 

\medskip

We triangulate $\Phi$ (Step 1) and introduce a new variable with rational constraint to deal with the inequations (Step 2), exactly as in proof of \cref{prop:MAIN-torsionfree}. 
From now on assume that the system $\Phi$ consists of $q\in \Oh(n)$ equations of the form
 $X_jY_j=Z_j$ where $1\leq j \leq q$. 
 
 \medskip
 
 \noindent\textbf{Steps 3 and 4: Lifting $\Phi$ to the virtually free group $V$}
 
 \medskip
 
\begin{itemize}
 \item[Step 3] Let $b \in \Oh(q)= \Oh(n)$ be the constant defined in \cref{prop:V} that depends on $\delta$ and linearly on $q$. 
 Recall that $$V_{\leq b}=\{[\gamma] \in V\mid \gamma \textrm{\ reduced\ and\ } \ell_{\mathcal{K}}(\gamma)\leq b \},$$ and let $Y$ be the generating set of $V$ as in \cref{std-DE-quasigeod}.
 
 One lifts the system $\Phi$ in $G$ to a finite set of systems $\Phi_{\mathbf c}$ 
 in the virtually free group $V$, one system for each $q$-tuple $\mathbf c=(c_{11},c_{12},c_{13},\dots, c_{q1},c_{q2},c_{q3})$ of triples $(c_{j1}, c_{j2}, c_{j3})$ with $c_{ij} \in V_{\leq b}$ and such that $f(c_{j1}c_{j2}c_{j3})=1_G$, as in \cref{prop:V}.
  We enumerate these tuples by enumerating triples of words $(v_{j1},v_{j2},v_{j3})$ over the generating set $Y$ of $V$ with $\ell_{(V,Y)}(v_{ij})\leq b_Y$, where $b_Y\in \Oh(q)$ depends on $b$, as in \cref{kappaV}(ii). 
 By \cref{kappaV}(ii) the set of corresponding path triples $(p_{v_{j1}},p_{v_{j2}},p_{v_{j3}})$ in $\mathcal{K}$ (see \cref{notationZ} -- \cref{unreduced}) contains all triples $(c_{j1}, c_{j2}, c_{j3})$ with $c_{ji} \in V_{\leq b}$, 
 up to homotopy. Then for each triple $(v_{j1},v_{j2},v_{j3})$ we check whether $f(v_{j1}v_{j2}v_{j3})=1_G$, which is done by checking whether in the Cayley graph $\Gamma(G,S)$ the identity $\gamma_{v_{j1}}\gamma_{v_{j2}}\gamma_{v_{j3}}=_G 1$ holds, using the Dehn algorithm in $G$.

 \item[Step 4]
 For each tuple $\mathbf c$ we construct a system $\Phi_{\mathbf c}$ of equations in $V$ of the form 
\begin{equation}\label{csystem}
X_j=P_jc_{j1}Q_j, \ Y_j=Q_j^{-1}c_{j2}R_j, \ Z_j=P_jc_{j3}R_j, \ \ 1\leq j\leq q,
\end{equation}
where $P_i, Q_i, Z_i$ are new variables. This new system has size $\Oh(n^2)$ since it has $\Oh(q)$($=\Oh(n)$) equations, each of length in $\Oh(q)$, and the $c_{ij}$'s inserted have length in $\Oh(q)$. 
 In order to avoid an exponential size complexity we write down each system $\Phi_{\mathbf c}$ one at a time, so the space required for this step is $\Oh(n^2)$.  
\end{itemize}

\medskip

\noindent \textbf{Steps 5 and 6: Solving equations in $V$ and then in $G$}

\medskip

Steps 5 and 6 produce an ET0L covering solution set of \qgeods \ for $\Phi$ in $G$. However, the technical Steps 5A and 6A are additionally needed in order to generate an EDT0L language of solutions in Theorem \ref{thmTorsion}.

 \begin{itemize}
 \item[Step 5][Solve in $V$] For each system $\Phi_{\mathbf{c}}$ over $V$ we run the modified DE algorithm (inserting the additional label $\mathbf c$ for each node printed) to print an NFA for each $\Phi_{\mathbf c}$. We obtain the set of solutions $\text{Sol}_{\langT,V}(\Phi_{\mathbf{c}})$ as an EDT0L language in $\NSPACE((q^2)^2\log (q^2))= \NSPACE(n^4\log n)$ of $(\lambda_Y,\mu_Y)$-\qgeods\ over $Y$. (These solutions are written in the ``standard normal forms" of \cite[Definition 14.1]{DEijac}, which are \qgeods\ by \cref{std-DE-quasigeod}.)

\item[Step 5A] [Duplicate solutions from $V$] This step is needed to get the solutions in $G$ as EDT0L at the very end; if omitted, the solutions will be ET0L only.

Apply the (Doubling) \cref{lem:doubleME} to obtain an EDT0L language of the form $\{p_1\#\dots \#p_rf(p_1\#\dots \#p_r)\}$ where $p_1\#\dots \#p_r \in \text{Sol}_{\langT,V}(\Phi_{\mathbf{c}})$ is a tuple of words in standard normal form and $f$ is the bijection $f:Y\to Y_\dag=\{y_\dag\mid y\in Y\cup\{\#\}\}$ as in the lemma. We can assume this language is
 \[L'\subseteq \left\{p_1\#\dots \#p_rq_1\#_\dag\dots \#_\dag q_r\mid p_i\in S^*, q_i\in S_\dag^*,\pi(p_i)=\pi_\dag(q_i)\right\},\]
and it consists of actual solutions $p_1\#\dots \#p_r$, together with their \emph{shadows} $q_1\#_\dag\dots \#_\dag q_r$. 
\item[Step 6] [Produce an alternative covering solution set of $\Phi_{\mathbf c}$ in $V$, then apply substitution map to words in this set to get quasigeodesic solutions in $G$] 

Let $\lambda'_1 \geq 1, \mu'_1\geq 0$ be the the constants in Lemma \ref{kappaV} (computable from $(\lambda_1,\mu_1)$, which were defined in (\ref{qgK})), and let $\langT'=Q_{V, Y, \lambda'_1, \mu'_1}$ be the set of $(\lambda'_1,\mu'_1)$-quasigeodesics in $V$ over $Y$. Then the set $\text{Sol}_{\langT',V}(\Phi_{\mathbf{c}})$ of all $(V,Y,\lambda'_1,\mu'_1)$-\qgeods \ which represent solutions of $\Phi_{\mathbf{c}}$ is ET0L by \cref{prop:shortlex-hyp} (1). Furthermore, by \cref{cor:KtoY} this set contains at least one word over $Y$ for each solution in $Q_{\mathcal{K}, V,\lambda_1, \mu_1}$.

Then for $\sigma$ as in (\ref{Gpath}), $\sigma(\text{Sol}_{\langT',V}(\Phi_{\mathbf{c}}))$ is ET0L since ET0L is preserved by substitutions, and by \cref{prop:V} the set 
 $\bigcup_{\mathbf{c}} \sigma(\text{Sol}_{\langT',V}(\Phi_{\mathbf{c}}))$ projects onto $\text{Sol}_{G}(\Phi)$,
 so it is a covering solution set of $\Phi$ in $G$. By \cref{VtoG} there exist constants $\lambda, \mu$ such that all words in $\bigcup_{\mathbf{c}} \sigma(\text{Sol}_{\langT',V}(\Phi_{\mathbf{c}}))$ are $(G,S,\lambda, \mu$)-\qgeod\, so we obtained an ET0L covering solution set for $\Phi$ contained in  $Q_{G,S,\lambda,\mu}$. 

\item[Step 6A] The operations in Step 6 are being performed not only on the actual solutions $\text{Sol}_{\langT,V}(\Phi_{\mathbf{c}})$, but also on their shadows (see Step 5A), that is, on the entire set $L'$, to obtain a set $L$ as required in the proposition. The solutions and their shadows are not in  bijection as words, but they represent the same tuples of elements in $V$ and then $G$, respectively. This will be sufficient for us to later apply \cref{prop:shortlex-hyp} to get \cref{thmTorsion}.

\end{itemize}
To see why the final statement of the proposition holds, note that when we call on the DE algorithm above, this will be able to report if the system over $V$ has infinitely many, finitely many, or no solutions. If DE reports finitely many or no solutions then we are done, as $L$ will be finite or empty, respectively. If DE returns infinitely many solutions, note that there are only finitely many words over $S$ per group element in $G$ in the solution set because only $(G,S,\lambda,\mu)$-\qgeods\ are part of output, so we end up with infinitely many solutions in $\text{Sol}_G(\Phi)$. 
\end{proof}

\begin{lemma} Let $\alpha \geq 1, \beta, L \geq 0$ be constants.
\label{kappaV}\leavevmode
\begin{enumerate}
\item[(i)] There exist $\alpha' \geq 1,\beta' \geq 0$ (computable from $\alpha,\beta$ and $Y$) such that for any $c \in Q_{\mathcal{K}, V, \alpha, \beta}$, there is a word $w$ on $Y$ representing $c$ such that $w \in Q_{V, Y, \alpha',\beta'}$.
\item[(ii)] There exists $L' \geq 0$ (computable from $L$ and $Y$) such that for any $c \in V$ with $\ell_{\mathcal{K}}(p_c)\leq L$ (recall that $p_c$ is the reduced path representing $c$ in $\mathcal{K}$),
there is a word $w \in Y^*$, $w=_V c$ with $\ell_{(V,Y)}(w)\leq L'$.  
\end{enumerate}
\end{lemma}
\begin{proof} 

Consider the generating set $Z=Y \cup V_{\leq 3}$ for $V$.

(i) Let $\gamma_v$ be the reduced path in $\mathcal{K}$ corresponding to some $v\in V$. We will construct a path $\gamma_v'$ in $\mathcal{K}$, homotopic to $\gamma_v$, such that $\ell_{\mathcal{K}}(\gamma_v) \leq \ell_{\mathcal{K}}(\gamma'_v) \leq 2\ell_{\mathcal{K}}(\gamma_v)$ and $\frac{1}{M}\ell_{\mathcal{K}}(\gamma_v)\leq |\gamma'_v|_Z \leq M\ell_{\mathcal{K}}(\gamma_v)$, where $M=\max\{\ell_{\mathcal{K}}(p_y) \mid y \in Y\}$ (that is, $M$ is the maximal length of a generator in $Y$ with respect to the associated reduced path length in $\mathcal{K}$).

The vertices in $\mathcal{K}$ belong either to $G$, call those $G$-vertices, or to $\mathcal{K}\setminus G$, call them $\neg G$-vertices. 
 Modify the path $\gamma_v$ as follows. If $\gamma_v$ contains subpaths that are in $\{p_y \mid y\in Y \} \cup V_{\leq 3}$ then leave them as they are. All other subpaths will have some maximal  sequences of $\neg G$-vertices $A_1, A_2, \dots A_t $, $t\geq 3$. 
 By \cref{rmkKG}, for any $\neg G$-vertex there is a $G$-vertex at distance $1$  in $\mathcal{K}$, so to every $A_i$ whose $\gamma_v$-neighbours are both in $\neg G$ choose some $G$-vertex $B_i$ at distance $1$ from $A_i$ in $\mathcal{K} \setminus \gamma_v $, and attach the backtrack $[A_iB_i,B_iA_i]$ to the path $\gamma_v$ to obtained the unreduced path $\gamma_v'$. Then $\gamma_v'$ is a concatenation of paths in $V_{\leq 3}$ and $\{p_y \mid y\in Y \}$, each of which is a generator of $V$ belonging to $Z$. This shows that one can obtain a path $\gamma'_v$ in $\mathcal{K}$ for $v$ which can be written as a word $w$ over $Z$, and $\frac{1}{M}\ell_{\mathcal{K}}(\gamma_v)\leq \abs{w}_Z \leq \ell_{\mathcal{K}}(\gamma_v) < M\ell_{\mathcal{K}}(\gamma_v)$. 

Now suppose $\gamma_c$ is an $(\alpha,\beta)$-\qgeod\ in $\mathcal{K}$. Then the modified path $\gamma'_c$ is homotopic to $\gamma_c$ and corresponds to a word $w$ over $Z$ as above; moreover, all subwords of $w$ have length proportional to the corresponding subpaths of $\gamma'_c$, so $w$ is a $(\alpha_M, \beta_M)-$\qgeod\ over $Z$, where $(\alpha_M, \beta_M)$ depend on $\alpha, \beta$ and $M$, and $M$ depends on $Y$. Since changing generating sets (from $Z$ to $Y$) is a quasi-isometry, $w$ is an $(V, Y, \alpha',\beta')-$\qgeod\ , and $(\alpha',\beta')$ depend on $Y$ and the initial $(\alpha,\beta)$.

(ii) This is an immediate application of (i).
\end{proof}

The next corollary ensures that by considering a sufficiently large set of quasigeodesics in $V$ over $Y$ we capture all the quasigeodesic solutions in $\mathcal{K}$ guaranteed to exist by Proposition \ref{prop:V} (1).
  
\begin{corollary}\label{cor:KtoY}
There exist $\lambda'_1, \mu'_1$ (depending on $\lambda_1, \mu_1$ as above, and $Y$) such that 
for any path $v \in Q_{\mathcal{K}, V, \lambda_1, \mu_1}$ there is a $(V,Y,\lambda'_1, \mu'_1)$-\qgeod\ word over $Y$ representing $v$. 
\end{corollary}

The following lemma shows that the substitution $\sigma: Y^* \mapsto S^*$ (\ref{Gpath}) will produce quasigeodesics in $G$ from an appropriate set of quasigeodesics in $V$, so the quasigeodesic solutions in $V$ given by \cite{DEijac} can be transformed into quasigeodesic solutions in the hyperbolic group $G$ in the proof of Proposition \ref{prop:MAIN-torsion}.

\begin{lemma}\label{VtoG}
Let $w \in Q_{V,Y, \lambda'_1, \mu'_1}$ be such that the reduced $\mathcal{K}$-path $p_{\pi(w)}$ is in $Q_{\mathcal{K}, V, \lambda_1, \mu_1}$. Then the unreduced path $p_{w}$ is \qgeod\ in $\mathcal{K}$ (for appropriate constants). 
In particular, there exist $\lambda \geq 1, \mu \geq 0$ (computable from $\lambda'_1, \mu'_1$ and $Y$ in Corollary \ref{cor:KtoY}) such that $\sigma(w) \in Q_{G, S, \lambda, \mu}$.
\end{lemma}

\begin{proof}
Consider the generating set $Z=Y \cup V_{\leq 3}$ for $V$ and let $\lambda_Z, \mu_Z$ be such that any $(V,Y,\lambda'_1, \mu'_1)$-\qgeod\ is $(V,Z, \lambda_Z, \mu_Z)$-\qgeod\ . 
We will show that for $(a_{\mathcal{K}},b_{\mathcal{K}})=(\lambda_1,\mu_1+M\mu_Z)$ the (unreduced) path $p_{w}$ is $(a_{\mathcal{K}},b_{\mathcal{K}})$-\qgeod\ in $\mathcal{K}$, where $M=\max\{\ell_{\mathcal{K}}(p_y) \mid y \in Y\}$, that is, $M$ is the maximal length of a generator in $Y$ with respect to the associated reduced path length in $\mathcal{K}$.

 We say that a subpath $s_w$ of $p_w$ is a maximal backtrack if $p_w=p s_w p'$, $s_w$ is homotopic to an empty path (via the elimination of backtrackings), and $s_w$ is not contained in a longer subpath of $p_w$ with the same property. This implies there is a point $A$ on $p_w$ such that $s_w$ starts and ends at $A$, and such a maximal backtrack traces a tree in $\mathcal{K}$.  
We can then write $p_w=a_1s_1 a_2 \dots s_{n-1}a_n$, where $a_i$ are (possibly empty) subpaths of $p_w$ and $s_i$ are maximal backtracks; thus $p_{\pi(w)}=a_1a_2 \dots a_n$. If $\ell_{\mathcal{K}}(s_i) \leq M\mu_Z$ for all $i$, then the result follows immediately. Otherwise there exists an $s_i$ with $\ell_{\mathcal{K}}(s_i)>M\mu_Z$, and we claim that we can write $s_i$ in terms of a word over $Z$ that is not a $(\mathcal{K},\lambda_Z, \mu_Z)$-\qgeod, which contradicts the assumption that $w$ is \qgeod.

To prove the claim, suppose $i(s_i)=f(s_i)=A$. We have two cases: in the first case $A\in G$ then $\pi(s_i)=_V 1$ and $s_i$ corresponds to a subword $v$ of  $w$ for which $\ell_{\mathcal{K}}(p_{v})\geq M\mu_Z$. But $v$ represents a word over $Z$, so $\ell_{\mathcal{K}}(p_{v}) \leq \ell_Z(v) M$, and altogether 
$ M\mu_Z \leq \ell_{\mathcal{K}}(p_{v}) \leq \ell_Z(v) M.$
Since $|v|_Z=0$ and $v \in Q_{V,Z,\lambda_Z, \mu_Z}$, $\ell_Z(v) \leq \mu_Z$, which contradicts $\ell_Z(v) \geq \mu_Z$ from above.

In the second case $A\notin G$, so take a point $B \in G$ at distance $1$ from $A$ in $\mathcal{K}$ (this can always be done), and modify the word $w$ to get $w'$ over $Z$ so that $p_{w'}$ in $\mathcal{K}$ includes the backtrack $[AB,BA]$ off the path $p_w$. Also modify $s_i$ to obtain a new backtrack $s'_i$. Clearly $\pi(p_w)=\pi(p_{w'})$ and $\pi(s_i)=\pi(s'_i)$, and $s'_i$ becomes a maximal backtrack of $p_{w'}$ which can be written as a word over the generators $Z$ that represents the trivial element in $V$. We can the apply then argument from the first case.

 Finally, recall from (\ref{Gpath}) that the path $\sigma(w)$ replaces each subpath $p_i \in \mathcal{K}$ (corresponding to a generator $z_i$) of $p_w$ by another path $\gamma_i$ with the same endpoints. and if we add midpoints to the edges in $\sigma(w)$ to get a path $\sigma'(w)$, we can view $\sigma'(w)$ as a path in $\mathcal{K}$, not just $\Gamma(G,S)$. Since by construction $\sigma'(w)$  and $p_w$ fellow travel, and $p_w$ is a \qgeod\ in $\mathcal{K}$, we get that $\sigma'(w)$ is \qgeod\ in $\mathcal{K}$, with constants depending on the fixed choice of $\gamma_i$'s from (\ref{Gpath}).
 Thus there exist $\lambda \geq 1,\mu \geq 0$ such that $\sigma(w)$ is a $(\lambda,\mu)$-\qgeod\ over $S$ in the hyperbolic group $G$ since $\mathcal{K}$ and $\Gamma(G,S)$ are quasi-isometric, and the path $\sigma(w)$ is obtained from the \qgeod\ $\sigma'(w)$ by removing the midpoints of edges.
\end{proof}

We can now prove our main result about  hyperbolic groups with torsion.
\begin{theorem}[Torsion]\label{thmTorsion}
Let $G$ be a hyperbolic group  with torsion, with finite symmetric generating set $S$.  Let $\#$ be a symbol not in $S$, and let $\Sigma=S\cup\{\#\}$.  
Let $\Phi$ be a system of equations and inequations
of size $n$ with constant size effective  \qier\ constraints, over variables $X_1,\dots,  X_r$, as in \cref{sec:notationSolns,defn:explicit}.

Then we have  the following.

 \begin{enumerate}
     \item\label{item:THMsurjectSet}  For any $\lambda'\geq 1,\mu'\geq 0, \lambda'\in\mathbb Q$ and for any regular language $\langT\subseteq Q_{G,S,\lambda', \mu'}$ such that the projection $\pi: \langT\to G$ is a surjection,
the   full set of $\langT$-solutions
 \[\text{Sol}_{\langT,G}(\Phi)=\{(w_1\#\dots\# w_r ) \mid w_i \in \langT,  (\pi(w_1),\dots, \pi(w_r) ) \text{ solves  } \Phi\}\]
   is ET0L  in $\NSPACE(n^4\log n)$.
  \item \label{item:THMshortlex}
  For any $\lambda'\geq 1,\mu'\geq 0, \lambda'\in\mathbb Q$ and for any regular language $\langT\subseteq Q_{G,S,\lambda', \mu'}$ such that the projection $\pi: \langT\to G$ is a bijection,
the   full set of $\langT$-solutions
 \[\text{Sol}_{\langT,G}(\Phi)=\{(w_1\#\dots\# w_r ) \mid w_i \in \langT,  (\pi(w_1),\dots, \pi(w_r) ) \text{ solves  } \Phi\}\]
   is EDT0L  in $\NSPACE(n^4\log n)$.
\item It can be decided in $\NSPACE(n^2\log n)$ whether or not the solution set of $\Phi$ is empty, finite or infinite.
\end{enumerate}
\end{theorem}
In particular, if $\langT$ is the set of all shortlex geodesics over $S$, then \Cref{thm:IntroTorsion}  follows immediately.

 \begin{proof}  The proof follows that of  the torsion-free case (Theorem \ref{thmTorsionFree}) with some necessary additional details. 
 Suppose $R_{X_i}$, $X_i\in \mathcal{X}$, are the  constant size \qier\ constraints that are part of the system $\Phi$.
 By  \cref{prop:DG9.4complexity}
 (\cite[Proposition 9.4]{DG})
 the languages $\widetilde{R_{X_i}}=\pi^{-1}(R_{X_i})\cap Q_{G,S,\lambda_G,\mu_G}$ are regular and one can compute automata accepting them, and as  before, this is done in constant space.

Let $\Phi'$ be the system consisting only of the equations and inequations from $\Phi$. 
Apply  \cref{prop:MAIN-torsion} to $\Phi'$ to obtain  an ET0L language 
\[L\subseteq \{w_1\#\dots\# w_rf(z_1\#\dots \#z_r)\mid w_i ,z_i\in Q_{G,S,\lambda,\mu}, \pi(w_i)=\pi(z_i), 1\leq i \leq r \}\] where $\{ w_1\#\dots\# w_r\mid w_1\#\dots\# w_rf(z_1\#\dots \#z_r)\in L\}$ is a covering solution to $\Phi'$.
Note that a key difference at this point to the torsion-free case is that we have doubled the language earlier, and our resulting language is not necessarily deterministic ET0L yet.

Intersect the language $L$ with the regular language 
$\widetilde{R_{X_1}}\#\dots\#\widetilde{R_{X_r}} S_\dag^*\#\dots \#S_\dag^*$  where $S_\dag$ is as in \cref{notation:doubleGset} and call the resulting language $L_c$.
The language $L_c$ is ET0L, as the intersection of an ET0L language with a regular language, and stays in the same space complexity of $\NSPACE(n^2\log n)$ by \cref{prop:closureET0L}(3).  
By construction we have \[ L_c\subseteq\{w_1\#\dots\# w_rf(z_1\#\dots \#z_r)\mid w_i ,z_i\in Q_{G,S,\lambda,\mu}, \pi(w_i)=\pi(z_i)\in R_{X_i}, 1\leq i \leq r \}\] and $\{ w_1\#\dots\# w_r\mid w_1\#\dots\# w_r(z_1\#\dots \#z_r)\in L_c\}$ is a covering solution to $\Phi$, that is, solves $\Phi'$ and obeys the constraints.

\cref{prop:shortlex-hyp} applied to $L_c$ shows  that
(1)  the set of solutions written as words in  a regular language  $\langT\subseteq Q_{G,S,\lambda', \mu'}$ surjecting onto the group is ET0L, and 
(2) the set of solutions written as words in a regular language $\langT\subseteq Q_{G,S,\lambda', \mu'}$  in bijection with the group is EDT0L.  Part (3) also follows immediately from \cref{prop:shortlex-hyp}.
\end{proof}

\section{Non-explicit constraints}\label{sec:constraints}

In this section we consider systems of equations with effective  \Qier\ constraints that are not necessarily  constant size with respect to the size of equations. In the proofs of Theorems \ref{thmTorsionFree} (torsion-free) and \ref{thmTorsion} (the torsion case), when we had to deal with constant size  \qier\ constraints, 
we applied 
\cref{prop:DG9.4complexity}
 (\cite[Proposition 9.4]{DG}) to replace the automaton 
for each rational constraint $R_{X_i}$ by another automaton accepting all $(G,S,\lambda,\mu)$-\qgeod\ words representing group elements which belong to the rational subset $R_{X_i}$; then we could intersect them with our covering solution set to obtain solutions which obey the constraints. Since the constraints were assumed to be constant size, this operation was considered to have no contribution to the complexity.
Now if we consider constraints whose size is considered as contributing to the input size, 
 \cref{prop:DG9.4complexity} shows that we have no hope in general to state any computable upper bound on the size required to obtain the automata  accepting $\widetilde{\mathcal R}_i=\pi^{-1}(R_{X_i})\cap Q_{G,S,\lambda,\mu}$.

We therefore state \cref{thmQIER} in terms of existence and decidability only, without any space complexity bound.

\begin{theorem}[Systems with arbitrary effective  \qier\ constraints]\label{thmQIER}
Let $G$ be a  hyperbolic group with or without torsion, with finite symmetric generating set $S$. 
 Let $\#$ be a symbol not in $S$. 
Let $\Phi$ be a  finite system of equations and inequations over variables $X_1,\dots,  X_r$, 
with  effective \qier\  constraints as in \cref{sec:notationSolns,defn:explicit}. 

Then we have  the following.
 \begin{enumerate}
     \item For any $\lambda'\geq 1,\mu'\geq 0, \lambda'\in\mathbb Q$ and for any regular language $\langT\subseteq Q_{G,S,\lambda', \mu'}$ such that the projection $\pi: \langT\to G$ is a surjection,
the   full set of $\langT$-solutions
 \[\text{Sol}_{\langT,G}(\Phi)=\{(w_1\#\dots\# w_r ) \mid w_i \in \langT,  (\pi(w_1),\dots, \pi(w_r) ) \text{ solves  } \Phi\}\]
   is ET0L.
  \item
  For any $\lambda'\geq 1,\mu'\geq 0, \lambda'\in\mathbb Q$ and for any regular language $\langT\subseteq Q_{G,S,\lambda', \mu'}$ such that the projection $\pi: \langT\to G$ is a bijection,
the   full set of $\langT$-solutions
 \[\text{Sol}_{\langT,G}(\Phi)=\{(w_1\#\dots\# w_r ) \mid w_i \in \langT,  (\pi(w_1),\dots, \pi(w_r) ) \text{ solves  } \Phi\}\]
   is EDT0L.
\item It is decidable whether or not the solution set of $\Phi$ is empty, finite or infinite. 

\end{enumerate}
\end{theorem}
\begin{proof} 
In the proofs of Theorems \ref{thmTorsionFree} (torsion-free) and \ref{thmTorsion} (the torsion case) the space  required by the algorithm which  constructs the automata accepting the languages $\widetilde{R_{X_i}}$ is no longer constant,  so do the same steps but ignore the space calculation.
Propositions~\ref{prop:MAIN-torsionfree} and ~\ref{prop:MAIN-torsion} produce an EDT0L (in the torsion-free case), respectively ET0L (in the torsion case) covering solution set to a system without any constraints (except say for the inequations). One then employs \cite[Proposition 9.4]{DG} to obtain finite state automata which encode the constraints for each variable as words over $Q_{G,S,\lambda',\mu'}$ for 
effectively computed values of $\lambda',\mu'$.
Intersect the EDT0L language with these languages 
 to obtain a covering solution set to the system $\Phi$ including constraints. Then apply \cref{prop:shortlex-hyp} to obtain the full set of $\langT$-solutions as E(D)T0L as appropriate. 

By \cite{Jones_Skyum_1977} it is decidable whether the language produced by some given ET0L system is  empty, finite, or infinite. So 
 it is decidable to conclude whether or not the solution set is empty, finite, or infinite since we have an effective construction of the NFA for the EDT0L grammar, and then we can decide finiteness/emptiness using \cite{Jones_Skyum_1977}. Equivalently, it follows from our construction and our previous work \cite{CDE,DEijac} that the solutions in the hyperbolic group are empty, finite or infinite if and only if the corresponding solution sets  in the associated (virtually) free group are.
\end{proof}

\section{Additional result for free groups}\label{sec:further}

In this section we prove \cref{cor:allwords}. This is a direct corollary of \cite{CDE} and does not rely on any other results in the present paper. 

\allFree*

\begin{proof}
Let $(C, A\cup\{\#\}, c_0, \mathcal M)$ be the EDT0L system from \cite{CDE}. Assume $c_0=a_1\dots a_s\in C^*$.
 We construct a new system $(C', A\cup\{\#\}, c_0', \mathcal M')$ where  $C'=C\cup\{\diamond\}$, $c_0'=\diamond a_1\diamond\dots \diamond a_s\diamond$ and $\mathcal M'$ is obtained from $\mathcal M$ as follows.
Replace each rule  $(a,u_1\dots u_k)$ in each table labeling an edge of $\mathcal M$ by 
$(a, u_1\diamond\dots \diamond u_k)$. 
Let $t$ be the  non-deterministic table $t=\{(\diamond,  \diamond), (\diamond, \diamond a \diamond a^{-1}\diamond)\mid a\in A\}$ (and is constant on all other letters of $C'$) which inserts cancelling pairs $aa^{-1}$ at certain places between $A$-letters in a word.
For each accept state $q$ of $\mathcal M$, print a new edge $(q,q_{\text{unreduce}}, \varepsilon)$ to a new state $q_{\text{unreduce}}$, and print a loop $(q_{\text{unreduce}},q_{\text{unreduce}},t)$.
Finally print an edge $(q_{\text{unreduce}}, q_{\text{accept}},\phi)$ to  a new unique accept state $q_{\text{accept}}$ where $\phi$ is the table (homomorphism) which sends $\diamond$ to the empty string.
 These modifications can be printed in the same space bound, and produce a grammar which accepts all possible tuples words obtained from tuples of shortlex words by inserting $aa^{-1}$ pairs (iterating the loop $t$ at $q_{\text{unreduce}}$ reverses the process of free reduction).\end{proof}

\section{Acknowledgements}

We sincerely thank Yago Antolin, Alex Bishop, 
 Fran\c{c}ois Dahmani,
  Volker Diekert, 
Derek Holt,
Jim Howie  and Alex Levine
 for invaluable help with this work. We also thank the anonymous reviewer for their careful reading and comments.

\bibliographystyle{plainurl}
\bibliography{Journal_Hyp}

\end{document}